\documentclass[11pt, oneside]{amsart}   	
\usepackage[margin = 1.2in]{geometry}     
\usepackage[mathscr]{euscript}
\usepackage{mathtools}
\usepackage{tikz, xcolor}
\usetikzlibrary{positioning,arrows}
\usetikzlibrary{decorations.markings}
\tikzstyle{circledvertex} = [draw, black, shape=circle, minimum size=8pt, inner sep=1pt]
\tikzstyle{invisivertex} = [black, shape=rectangle, minimum size=0pt, inner sep=2pt]
\tikzstyle{point}=[draw, black, fill,shape=circle, minimum size=4pt, inner sep=0pt]
\tikzstyle{label} = [draw, violet, shape=circle, minimum size=8pt, inner sep=1pt]
\tikzstyle{glabel} = [draw, teal, shape=circle, minimum size=4pt, inner sep=1pt]
\tikzstyle{graylabel} = [draw, gray, shape=circle, minimum size=4pt, inner sep=1pt]

\tikzstyle{BVertex}  = [draw, black, fill, shape=circle, minimum size=3pt, inner sep=0pt]

\usepackage{graphicx}	
\usepackage{subcaption}
\newtheorem{theorem}[subsection]{Theorem}

\newtheorem{cor}[subsection]{Corollary}

\newtheorem{prop}[subsection]{Proposition}

\theoremstyle{definition}
\newtheorem{definition}[subsection]{Definition}
\newtheorem{remark}[subsection]{Remark}
\newtheorem{example}[subsection]{Example}
\newtheorem{question}[subsection]{Question}
\usepackage{tikz-cd}
\usepackage{bbm}
\usepackage{amssymb}
\usepackage{hyperref}
\DeclareMathOperator{\Z}{\mathbb{Z}}
\DeclareMathOperator{\GL}{GL}
\DeclareMathOperator{\R}{\mathbb{R}}
\DeclareMathOperator{\C}{\mathbb{C}}
\DeclareMathOperator{\Q}{\mathbb{Q}}
\DeclareMathOperator{\ch}{\mathcal{C}}

\DeclareMathOperator{\D}{\mathcal{D}}

\DeclareMathOperator{\s}{\mathcal{S}}
\DeclareMathOperator{\A}{\mathcal{A}}
\DeclareMathOperator{\E}{\mathcal{E}}

\DeclareMathOperator{\lat}{\mathcal{L}}
\DeclareMathOperator{\F}{\mathcal{F}}
\DeclareMathOperator{\I}{\mathcal{I}}

\DeclareMathOperator{\T}{\mathcal{T}}

\DeclareMathOperator{\pconf}{PConf}

\DeclareMathOperator{\codim}{codim}
\DeclareMathOperator{\U}{\mathcal{U}}
\DeclareMathOperator{\uu}{\mathfrak{u}}
\DeclareMathOperator{\tpart}{\tilde{\partial}}

\DeclareMathOperator{\Det}{det}

\DeclareMathOperator{\VG}{\mathcal{V}}
\newcommand{\com}{\mathcal{M}}
\newcommand{\oddcom}{\mathcal{M}^{d}}

\newcommand{\dopp}{\mathcal{D}^\mathrm{opp}}
\newcommand{\Oe}{\overline{e}}

\DeclareMathOperator{\sgn}{sgn}

\DeclareMathOperator{\des}{des}
\DeclareMathOperator{\Des}{Des}
\DeclareMathOperator{\WH}{WH}
\DeclareMathOperator{\hilb}{Hilb}

\DeclareMathOperator{\ind}{Ind}

\usepackage{amsmath}

\def\frake{{\mathfrak{e}}}

\subjclass[2010]{05Exx, 20F55, 55R80, 14N20, 52C35}

\keywords{Eulerian idempotents, hyperplane arrangements, Varchenko-Gelfand ring, configuration spaces, Solomon's descent algebra, reflection groups, Bidigare-Hanlon-Rockmore}

\title{Eulerian representations for real reflection groups}
\author{Sarah Brauner}
\address{Department of Mathematics, University of Minnesota, Twin Cities, MN}
\email{braun622@umn.edu}

\begin{document}
\maketitle
\begin{abstract}
The Eulerian idempotents, first introduced for the symmetric group and later extended to all reflection groups, generate a family of representations called the Eulerian representations that decompose the regular representation. In Type $A$, the Eulerian representations have many elegant but mysterious connections to rings naturally associated with the braid arrangement. Here, we unify these results and show that they hold for any reflection group of \emph{coincidental type}---that is, $S_{n}$, $B_{n}$, $H_{3}$ or the dihedral group $I_{2}(m)$---by giving six characterizations of the Eulerian representations, including as components of the associated graded Varchenko-Gelfand ring $\mathcal{V}$. As a consequence, we show that Solomon's descent algebra contains a commutative subalgebra generated by sums of elements with a fixed number of descents if and only if $W$ is coincidental. 
More generally, for any finite real reflection group, we give a case-free construction of a family of Eulerian representations described by a flat-decomposition of the ring $\mathcal{V}$.

\end{abstract}

\section{Introduction}\label{sec:intro}
This paper studies two related families of orthogonal idempotents
within the group algebra $\R W$ of any {\it finite reflection group} $W$, that decompose the regular representation into $W$-representations recurring many times in the literature.

Recall that a reflection group is a finite subgroup $W$ of the general linear group $GL(V)$ for $V=\R^r$, generated by orthogonal {\it reflections} through various {\it reflecting hyperplanes} $H$, each of which is a codimension one linear subspace of $V$.  One then has an associated {\it reflection hyperplane arrangement} $\A=\{H_i \}_{i \in I}$,
and its {\it intersection lattice} $\lat(\A)$, which is simply the collection of all intersection subspaces $X=H_1 \cap \cdots \cap H_m$ of subsets of the hyperplanes.  Work of Saliola \cite{sal2008quiv,saliola2009face,saliola2012eigenvectors}, reviewed in Section \ref{sec:eulhyp} below, associates to each such intersection $X$ an idempotent $\frake_X$ in the \emph{face (Tits) algebra} of $\A$, and
$\{ \frake_X \}_{X \in \lat(A)}$ turn out to give a complete family of orthogonal idempotents for this algebra;
we call these {\it flat idempotents}\footnote{The family of idempotents depends on a choice of \emph{section map}, also to be defined in Section \ref{sec:eulhyp}.} of $\A$.
We group them further into two coarser complete families
of orthogonal idempotents.  Letting
\[
[X]:=\{ Y=wX: w \in W\} \subset \lat(\A)
\]
denote the $W$-orbit of the intersection space $X$, 
we will consider the idempotents
\[
\frake_{[X]} :=\sum_{Y \in [X]} \frake_Y 
\]
as $[X]$ runs through the
$W$-orbits $\lat(\A)/W$  on $\lat(A)$, which we call \emph{flat-orbit idempotents}. The $\frake_{[X]}$ can be realized as idempotents in $\R W$ via a result of Bidigare \cite{bidigare}, and in this case they correspond to idempotents introduced by F. Bergeron, N. Bergeron, Howlett and Taylor in \cite{bergeron1992decomposition}.
There are even coarser
idempotents
\[
\frake_{k} :=\sum_{\substack{Y \in \lat(\A)\\ \dim(Y)=k}} \frake_Y 
\]
for $k=0,1,\dots,r$. 
This last family will be called the {\it Eulerian idempotents} for $W$ and can also be realized in $\R W$. 

Our goal in this paper is to analyze two families of representations.  
First, the \emph{Eulerian representations} $\{ \R W \frake_{k} \}_{0 \leq k \leq r}$ when $W$ is a reflection group of \emph{coincidental type}\footnote{These groups are called \emph{good reflection groups} by Aguiar-Mahajan in \cite{aguiarmahajan}.}; that is, an irreducible finite real reflection group of rank $r$ whose exponents (equivalently, degrees) can be expressed in terms of an \emph{exponent gap} $g$: 
\[1, 1+g, 1+2g, \dots, 1+ (r-1)g. \]
These are exactly reflection groups of Types $A$ and $B$, $H_{3}$, and the dihedral group $I_{2}(m)$. Second, we study the family of representations $\{ \R W \frake_{[X]} \}_{[X] \in 
 \lat(\A)/W}$ induced by the flat orbit idempotents for any real finite reflection group $W$.

\subsection*{Motivating story: Type $A$}
The Eulerian idempotents $\frake_{k}$ described above generalize the \emph{Type A Eulerian idempotents}, which have been studied extensively beginning in the late 1980s, when they were introduced independently by both Reutenauer in \cite{reutenauer1986theorem} and Gerstenhaber-Schack in \cite{Gerstenhaber-Schack}. 

Reutenauer introduced the idempotents $\frake_{k}$ in $\R S_{n}$ as part of his work on the Campbell-Baker-Hausdorff formula. In \cite{garsia}, Garsia and Reutenauer showed that this family of idempotents could be defined via the generating function
\begin{equation}\label{eq:typeaeul} \sum_{k=0}^{n-1} t^{k+1}\frake_{k} = \sum_{w \in S_{n}} \binom{t-1+n-\des(w)}{n}w,\end{equation}
where one defines the \emph{Coxeter descent set} for any Coxeter system $(W,S)$,
\begin{equation}
    \label{descent-set-definition}
\Des(w) := \{ s \in S: \ell(w) > \ell(ws) \}
\end{equation}
and the {\it descent number} 
\[
\des(w) = |\Des(w)|.
\]

By contrast, Gerstenhaber and Schack were interested in giving a Hodge-type decomposition of Hochschild homology, a homology theory for associative algebras. Earlier in \cite{barr}, Barr had defined a ``shuffle product'' $\s(S_{n})$ (\emph{Barr's element}), which can be phrased in the language of descents as\footnote{Barr's element was originally defined by tensoring $\s(S_{n})$ as defined above with the sign representation.}
\[ \s(S_{n}):= \sum_{s \in S} \sum_{\substack{w \in S_{n} \\ \Des(w) \subset \{s \}}} w.\]
Gerstenhaber and Schack built upon Barr's work, proving that $\s(S_{n})$ 1) acts semisimply on $\R S_{n}$ with eigenvalues $\sigma_{k} = 2^{k+1}-2$ for $0 \leq k \leq n-1$ and 2) commutes with the Hochschild boundary operator. Using Lagrange interpolation, they constructed a family of idempotents that are polynomials in $\s(S_{n})$ and for each $k$, project onto the $\sigma_{k}$-eigenspace of $\s(S_{n})$. While it is not obvious that these viewpoints should yield the same results, in \cite{lodayoperations}, Loday shows that these idempotents are precisely the $\frake_{k}$ in \eqref{eq:typeaeul}. It is likewise not immediately apparent that the Saliola construction of the Eulerian idempotents is consistent with either of the above definitions; our work will help elucidate these equivalences.\footnote{In addition, see Aguiar-Mahajan \cite[Sec. 16.11-16.12]{aguiarmahajan} and the references therein.} 

For our purposes, perhaps the most interesting aspect of the Type $A$ Eulerian idempotents are the properties of the $S_{n}$ representations $\R S_{n} \frake_{k}$. In the $k=0$ case, $\R S_{n} \frake_{k} \otimes \sgn_{S_{n}}$ is isomorphic to the top homology of the partition lattice $\Pi_{n}$ (see Barcelo \cite{barcelo1990action}, Joyal \cite{joyal1986foncteurs}, Wachs \cite{wachs1998co}), and $\R S_{n} \frake_{k}$ is isomorphic to the multilinear component of the free Lie algebra on $n$ generators (see Garsia \cite{garsia1990combinatorics}, Reutenauer \cite{reutenauer2001free}). 

Even more surprising is the following ``folklore'' fact: 
\[ \R S_{n} \frake_{k} \cong_{S_{n}} H^{(n-k-1)d}(\pconf_{n}(\R^{d}); \R) \] when $d \geq 3$ and odd, where $\pconf_{n}(\R^{d})$ is the space of $n$ distinct labeled points in $\R^{d}$. This can be deduced by comparing a result of Sundaram and Welker for subspace arrangements \cite[Thm 4.4(iii)]{sundaramwelker} with descriptions of the characters of $\R S_{n} \frake_{k}$ by Hanlon \cite{Hanlon}; see Sundaram \cite[Sec. 2: Thm. 2.2, Eq. 23]{sundaram2018} for history, or Early-Reiner \cite[Eq. 1.1]{Early-Reiner}. The space $H^{*}\pconf_{n}(\R^{d})$ is well-studied and connects the $\frake_{k}$ to other rings associated with the Braid arrangement (to be discussed shortly). These Type $A$ properties are the inspiration for our results.

\subsection*{Hint at a more general phenomenon: Type $B$}
As in Type $A$, the work of Aguiar and Mahajan generalizes earlier work by F. Bergeron and N. Bergeron in \cite{Bergeron-Bergeron} for Type $B$. Like Garsia and Reutenauer, Bergeron and Bergeron define the \emph{Type $B$ Eulerian idempotents} as elements in $\R B_{n}$ using the generating function\footnote{The idempotents that Bergeron and Bergeron define are actually obtained by tensoring the $\frake_{k}$ in \eqref{eq:typebeul} with the sign representation.}
\begin{equation}\label{eq:typebeul}
    \sum_{k=0}^{n}t^{k} \frake_{k} = \sum_{w \in B_{n}} \binom{\frac{t-1}{2} + n - \des(w)}{n} w.
\end{equation}
Like Gerstenhaber and Schack, they show that the $\frake_{k}$ give a Hodge decomposition of Hochschild homology for a commutative hyperoctahedral algebra\footnote{A hyperoctahedral algebra is an algebra with an involutive automorphism.}, although they do not use a Barr-like element to do so.

In \cite{bergeron1991hyperoctahedrallie}, N. Bergeron gives a description of the $B_{n}$ representation $\R B_{n} \frake_{0} \otimes \sgn_{B_{n}}$ as the top homology of the intersection lattice for the Type $B$ hyperplane arrangement\footnote{See Gottlieb-Wachs \cite{wachsdowling} for an alternate proof of this fact.}---thus hinting that the features of the Eulerian representations in Type $A$ might hold more generally. We will show this to be true.

\subsection*{Methods}
Our aim is to describe the Eulerian representations in terms of three closely related spaces:
\begin{enumerate}
    \item The \emph{associated graded Varchenko-Gelfand ring} $\VG$ (to be defined in Section \ref{sec:vargel}, Definition \ref{def:vg}): intuitively, $\VG$ can be thought of as a commutative version of the (better studied) Orlik-Solomon algebra;
    \item The cohomology\footnote{Henceforth, all cohomology and homology groups are assumed to have coefficients in $\R$.} of a ``$d$-dimensionally thickened'' hyperplane complement 
    \[\oddcom_{\A}:= V \otimes \R^{d} - \Big( \bigcup_{H_{i} \in \A} H_{i} \otimes \R^{d} \Big); \textrm{ and} \]
    \item The homology of open intervals $(V,X)$ in $\lat(\A)$:  for each $X$ in $\lat(\A)$, the set-wise $W$-stabilizer subgroup $N_X$ acts on the order complex $\Delta(V,X)$ and on its homology $H_*(V,X)$, which is nonvanishing only in degree $\codim(X)-2$. We will abbreviate the name of its $N_X$-representation\footnote{The notation here refers to the fact that $WH_{X}$ is a summand of \emph{\underline{W}hitney \underline{H}omology}; see Section \ref{sec:WHandGM}.} as
    $$
    \WH_{X}:= H_{\codim(X)-2}(V,X)
    $$
    and define from it an induced $W$-representation
    $$
    \WH_{[X]}:= \ind_{N_{X}}^{W} \WH_{X} \otimes \Det_{V/X}.
    $$
    where $\Det_{V/X}(w)$ is the determinant of $w \in N_{X}$ acting on $V/X$.
\end{enumerate}
The relationship between the associated graded Varchenko-Gelfand ring and Orlik-Solomon algebra is best understood through $\oddcom_{\A}$: in the $d = 2$ case, $\mathcal{M}^{2}_{\A}$ is the complexification of the hyperplane complement $\mathcal{M}^{1}_{\A}$, and $H^{*}(\mathcal{M}^{2}_{\A})$ is (equivariantly) isomorphic to the Orlik-Solomon algebra as a graded ring. A recent result of Moseley in \cite{moseley2017equivariant} shows that when $d \geq 3$ and odd, $H^{*}(\oddcom_{A})$ (equivariantly) describes $\VG$ as a graded ring. 

In the case of the Braid arrangement, $\oddcom_{\A(S_{n})} = \pconf_{n}(\R^{d})$ and there is a description of the cohomology due to F.~ Cohen \cite{cohen} for $d$ of any parity. Similarly, in Type $B$, $\oddcom_{\A(B_{n})}$ is $\pconf_{n}^{\Z_{2}}(\R^{d})$, the $\Z_{2}$ \emph{orbit configuration space} (see Section \ref{sec:config}, Definition \ref{def:orbconfig}) with cohomology presentation given by Xicotencatl \cite{xico} for any $d$.

Our contribution will be to connect all of these spaces---which already have well-known relationships to each other---to the Eulerian idempotents (in all of their guises). In doing so, we will avoid any character computations and rather tie together various equivariant versions of results in the literature, such as work by Aguiar-Mahajan \cite{aguiarmahajan}, Reiner-Saliola-Welker \cite{RSW}, and Sundaram-Welker \cite{sundaramwelker}. The main novelties in our methods are 1) to define generalizations and extensions of Barr's element and study their action on $\R W$ and 2) to further analyze the associated graded Varchenko-Gelfand ring in order to use it as a stepping-stone between other spaces. 

\subsection*{Results for coincidental reflection groups}
In the case of coincidental reflection groups, our primary tool will be a generalization of Barr's element in $\R W$. Define
\[ \s(W):= \sum_{s \in S} \sum_{\substack{w \in W \\ \Des(w) \subset \{s \}}} w.\]
We will show that $\s(W)$ acts semisimply on $\R W$, and when $W$ is coincidental has eigenvalues $\sigma_{0}< \sigma_{1} < \dots < \sigma_{r}$ in $\Z_{\geq 0}$, where $\sigma_{k}$ counts the number of rays (i.e. halfspaces for lines $L$ in $\lat(\A)$) lying in any intersection space $X$ with $\dim(X) = k$.

As a consequence, we are able to determine when the \emph{Eulerian subspace} 
\[ \E(W):=  \Big\{  \sum_{w \in W} c_{w} w : c_{w} = c_{w'} \textrm{ if }  \des(w) = \des(w')  \Big\} \subset \R W \]
is a commutative subalgebra, as it is in Types $A$ and $B$ (see Garsia-Reutenauer \cite{garsia} and Bergeron-Bergeron \cite{Bergeron-Bergeron}).
\begin{theorem}[Theorem \ref{thm:eulsubalg}]
The Eulerian subspace $\E(W)$ is a subalgebra if and only if $W$ is coincidental. Moreover, when the Eulerian subalgebra exists, it is always commutative. 
\end{theorem}

Let 
\[ \beta_{W,\ell}:= \frac{1}{|W|} \prod_{i=1}^{\ell} (t-e_{i}) \prod_{i=1}^{r-\ell} (t+e_{i}).\]

Our main theorem is a description of the Eulerian representations.

\begin{theorem}[Theorem \ref{thm:coincidental}]
When $W$ is a coincidental reflection group of rank $r$, for each $0 \leq k \leq r$, the following are equivalent as $W$-representations:
\begin{enumerate}
    \item The $k$-th graded piece of the associated graded Varchenko-Gelfand ring, $\VG^{k}$;
    \item  $H^{k(d-1)}(\oddcom_{\A})$ for $d \geq 3$ and odd;
    \item 
        $\bigoplus_{[X]} \WH_{[X]}$, where the direct sum is over all $[X] \in \lat(\A)/W$ with $\codim(X) = k$;
    \item The $\sigma_{r-k}$ eigenspace of $\s(W)$ in $\R W$;
    \item  The left $\R W$-module $\R W \mathfrak{e}_{r-k}$;
    \item  The left $\R W$-module $\R W E_{r-k}$, 
    where $\{E_k\} \subset \E(W)$ are idempotents defined by
        \[ \sum_{k = 0}^{r} t^{k}E_{k} := \sum_{w \in W} \beta_{W,\des(w)}(t) \cdot w. \]
\end{enumerate}
\end{theorem}

Theorem \ref{thm:coincidental} recovers all known descriptions of the Type $A$ and $B$ Eulerian representations, and also implies that the Type $B$ Eulerian representations are isomorphic to the non-trivial pieces of $H^{*}\pconf_{n}^{\Z_{2}}(\R^{d})$ for $d \geq 3$ and odd (Corollary \ref{cor:typebconfig}).

\subsection*{Results for arbitrary finite reflection groups}
We then study the representations $\R W \frake_{[X]}$ for any finite Coxeter group $(W,S)$. As in the coincidental case, we define an element $\T \in \R W$ and show that $\T$ acts semisimply with eigenspaces indexed by flat orbits $[X] \in \lat(\A)/W$. In Theorem \ref{thm:vggrading}, we show that $\VG$ admits a grading by flats. Theorem \ref{thm:allcox} then gives a case-free description of the representation carried by each eigenspace of $\T$ indexed by $[X] \in \lat(\A)/W$ in terms of 
\begin{itemize}
    \item a direct summand $\VG_{[X]}$ of the ring $\VG$ indexed by $[X]$,
    \item the representation $\WH_{[X]}$, and
    \item the representation  $\R W \mathfrak{e}_{[X]}$ generated by a flat-orbit idempotent $\mathfrak{e}_{[X]}$.
\end{itemize}

\subsection*{Outline} 
The remainder of the paper proceeds as follows. Section \ref{sec:hyperplanes} gives relevant background on the theory of hyperplane arrangements; Section \ref{sec:topology} covers necessary topological descriptions of subspace arrangements and defines the Varchenko-Gelfand ring; Section \ref{sec:coincidental} examines the case that $W$ is a coincidental reflection group; Section \ref{sec:allcox} addresses the case that $W$ is any finite reflection group; Section \ref{sec:future} proposes directions for future study.

\section*{Acknowledgements}
The author is very grateful to Victor Reiner for guidance at all stages of this project, to Marcelo Aguiar for illuminating discussions, to Franco Saliola for context on the history of the Eulerian idempotents, to Sheila Sundaram for suggestions on improvements to the exposition, and to Fran\c cois Bergeron, Nantel Bergeron, Christin Bibby, Dan Cohen, Galen Dorpalen-Barry, Alex Miller and Christophe Reutenauer for helpful references. The author also thanks an anonymous referee for their extremely insightful and helpful feedback. The author is supported by the NSF Graduate Research Fellowship (Award Number DMS-0007404).

\section{Real Hyperplane arrangements}\label{sec:hyperplanes}
The theory of real hyperplane arrangements underpins much of this paper; in this section we review the essential details. For a more in-depth study, consult \cite{aguiarmahajan}. 
\subsection{Basics}

A \emph{central hyperplane} $H$ in a vector space $V = \R^{r}$ is a codimension one subspace of $V$; a collection of hyperplanes is a \emph{hyperplane arrangement}, $\A$. All hyperplane arrangements we will consider will be \emph{central} and \emph{essential}, meaning that $\bigcap_{H \in \A} H = \{ 0 \}$ is the \emph{center} of $\A$.

A \emph{flat} $X$ of $\A$ is a subspace of $V$ formed by intersecting a subset of the hyperplanes in $\A$; when $\A$ is essential and central, $V$ and $\{ 0 \}$ are always flats. The collection of all flats of $\A$, ordered by reverse inclusion, defines a lattice $\lat(\A)$ (written $\lat$ when the context of $\A$ is clear.) This lattice is geometric, meaning that it is atomic and semimodular, with minimum element $V$ and maximum element $\{0 \}$. The rank function of $\lat$ is defined by codimension, so that $X \in \lat$ has rank $\codim(X):= r- \dim(X)$; because $\A$ is essential the rank of $\lat$ is $\dim(V)=r.$ Throughout this paper, $\mu$ will denote the M\"{o}bius function of $\lat$.  

For any flat $X \in \lat$, it is possible to define two new hyperplane arrangements. The \emph{localization arrangement of $X$}, $\A_{X}$ is 
\[ \A_{X}:= \{ H: H \in \A \text{ and } X \subset H \}. \]
In $\A_{X}$, $X$ is the maximal element of $\lat(\A_{X})$ and $V$ is the minimal element.
The \emph{restriction arrangement of $X$}, $\A^{X}$ is
\[ \A^{X}:= \{ H  \cap X \in \A: X \not \subset H \}. \]
In this case, the maximal element of $\lat(\A^{X})$ is $\{ 0 \}$ and the minimal element is $X$.

\subsection{The Tits algebra}\label{sec:titsalg}
Every hyperplane $H \in \A$ defines two disjoint open half spaces $H^{+}$ and $H^{-}$ in $V$ with respective closures $\overline{H^{+}}$, $\overline{H^{-}}$, so that  $H = \overline{H^{+}} \cap \overline{H^{-}}$. 
Index the hyperplanes in $\A$ with a set $I$. A \emph{face} of $\A$ is an intersection of the closure of half-spaces
$F = \cap_{i \in I} \overline{H^{\epsilon_{i}(F)}}$, where $\epsilon_{i}(F) \in \{ +, -, \pm \}$ and $\overline{H^{\pm}} := H$. Let $\F = \F(\A)$ be the set of faces of $\A$.
For our purposes, we will say two arrangements $\A, \A^{'}$ are isomorphic if there is a poset isomorphism between the set of faces $\F(\A)$ and $\F(\A^{'})$, ordered by reverse inclusion. 

The \emph{support} of a face $F$ is the smallest flat containing $F$; the \emph{support map} 
\[\mathfrak{s}: \F \to \lat \]
sends each face to its support. The dimension of $F$ is the (vector-space) dimension of $\mathfrak{s}(F)$. A face of maximal dimension (or equivalently, a face whose support is $V$) is called a \emph{chamber}. Let $\ch(\A) = \ch$ be the set of chambers of $\A$.  

The set $\F$ has a semigroup structure, with a product called the \emph{Tits product}.
\begin{definition}[Tits product]
For $F, G \in \F$, let 
\[ \epsilon_{i}(FG) = \begin{cases} \epsilon_{i}(F) &  \epsilon_{i}(F) \neq \pm\\
\epsilon_{i}(G) &  \epsilon_{i}(F) = \pm.
\end{cases}\]
The product $FG$ is then defined to be
\[ FG:= \bigcap_{i \in I}  \overline{H^{\epsilon_{i}(FG)}}. \]
\end{definition}

The Tits product has a geometric interpretation as well. For $F, G \in \F$, 
$FG$ is the first face one enters when moving an infinitesimally small (but nonzero) distance along the straight line segment from any point in the interior of $F$ towards any point in the interior of $G$. 

\begin{example}[Braid arrangement]\label{ex:braid1}
Perhaps the most well-studied hyperplane arrangement is the Braid arrangement $\A(S_{n})$, which consists of hyperplanes $H_{ij}:= \{ x_{i} - x_{j} = 0 \}$ and reflections $(ij) \in S_{n}$ over the $H_{ij}$.
The $H_{ij}$ are defined in $\R^{n}$, but to make the arrangement essential, we must project them into $\R^{n}/ \langle x_{1} + x_{2} + \dots + x_{n} \rangle \cong \R^{n-1}$. In Figure \ref{fig:braid}, the graphic on the left shows the essentialized arrangement $\A(S_{3})$ in $\R^{2}$.

The faces of $\A(S_{3})$ are shown in the graphic on the right of Figure \ref{fig:braid}. The origin (or center), written $\mathcal{O}$ is defined by 
\[\mathcal{O} = \overline{H_{12}^{\pm}} \cap \overline{H_{23}^{\pm}} \cap \overline{H_{13}^{\pm}}. \]
The chambers are $c_{1}, \dots, c_{6}$; the chamber $c_{1} = \overline{H_{12}^{+}} \cap \overline{H_{23}^{+}} \cap \overline{H_{13}^{+}}$, and the other $c_{i}$ can be expressed similarly. There are six one-dimensional faces, $f_{1}, f_{2}, \dots, f_{6}$, which are rays with starting point $\mathcal{O}$. The face $f_{1}$ is defined by $f_{1}= \overline{H_{12}^{\pm}} \cap \overline{H_{23}^{+}} \cap \overline{H_{13}^{+}}$. Note that smaller dimensional faces are contained in larger dimensional faces; for example, $\mathcal{O} \subset f_{1} \subset c_{1}$. Under the Tits product, every face in $\A(S_{3})$ is idempotent, and $\mathcal{O}$ is the identity. Other example computations include: 
\[ f_{1}c_{1} = c_{1}, \hspace{1.5em} f_{1}c_{5} = c_{1}, \hspace{1.5em} c_{1}c_{6} = c_{1}, \hspace{1.5em}f_{1}f_{2} = c_{2}. \]

\begin{figure}[!h]
\begin{center}
  \begin {tikzpicture}[scale=.75]

  \node[invisivertex] (A1) at (-1.5,-2.598){};
  \node[invisivertex] (A2) at (1.5,2.598){$H_{12}$};
 \node[invisivertex] (A3) at (-1,-1.9){};
 \node[invisivertex] (A4) at (0,-2.45){};

  \node[invisivertex] (B1) at (1.4,-2.598){$H_{13}$};
  \node[invisivertex] (B2) at (-1.5,2.598){};
   \node[invisivertex] (B3) at (1,-1.598){};
   \node[invisivertex] (B4) at (1.75,-1.15){};

  \node[invisivertex] (D1) at (-3,0){};
  \node[invisivertex] (D2) at (3,0){$H_{23}$};
  \node[invisivertex] (D3) at (2,.75){};
  \node[invisivertex] (D4) at (2,0){};

    \node[invisivertex] (f1) at (0,0){};

  \path[-] (A1) edge [bend left =0] node[above] {} (A2);
  \path[-] (B1) edge [bend left =0] node[above] {} (B2);
  \path[-] (D1) edge [bend left =0] node[above] {} (D2);
    \path[->] (D4) edge [ thick] node[left] {\tiny{$H_{23}^{+}$}} (D3);
\path[->] (A3) edge [ thick] node[above] {\tiny{$H_{12}^{+}$}} (A4);
\path[->] (B3) edge [ thick] node[below] {\tiny{$H_{13}^{+}$}} (B4);

  \end{tikzpicture}
  \hspace{3em}
    \begin {tikzpicture}[scale=.75]
  \node[BVertex] (f1) at (0,0){};

  \node[invisivertex] (A1) at (-1.5,-2.598){};
  \node[invisivertex] (A2) at (1.5,2.598){};
  \node[invisivertex] (B1) at (1.5,-2.598){};
  \node[invisivertex] (B2) at (-1.5,2.598){};
  \node[invisivertex] (D1) at (-3,0){};
  \node[invisivertex] (D2) at (3,0){};

  \node[label] (l1) at (1.73,1){$c_{1}$};
    \node[label] (l7) at (0,2){$c_{2}$};
    \node[label] (l6) at (-1.73,1){$c_{3}$};
    \node[label] (l5) at (-1.73,-1){$c_{4}$};
  \node[label] (l3) at (0,-2){$c_{5}$};
  \node[label] (l2) at (1.73,-1){$c_{6}$};
    \node[glabel] (a1) at (-1.75,-3.03){\small{$f_{4}$}};
  \node[glabel] (12) at (1.75,3.03){\small{$f_{1}$}};
  \node[glabel] (b1) at (1.75,-3.03){\small{$f_{5}$}};
  \node[glabel] (b2) at (-1.75,3.03){\small{$f_{2}$}};
  \node[glabel] (d1) at (-3.5,0){\small{$f_{3}$}};
  \node[glabel] (d2) at (3.5,0){\small{$f_{6}$}};
  \node[graylabel] (f2) at (.5,.3) {\tiny{$\mathcal{O}$}};

  \path[-] (A1) edge [bend left =0] node[above] {} (A2);
  \path[-] (B1) edge [bend left =0] node[above] {} (B2);
  \path[-] (D1) edge [bend left =0] node[above] {} (D2);
  \end{tikzpicture}

\end{center}
\caption{On the left, the Type $A$ hyperplane arrangement $\A(S_{3})$ with the positive half-spaces shown. 
On the right, the faces of $\A(S_{3})$.}\label{fig:braid}
\end{figure}
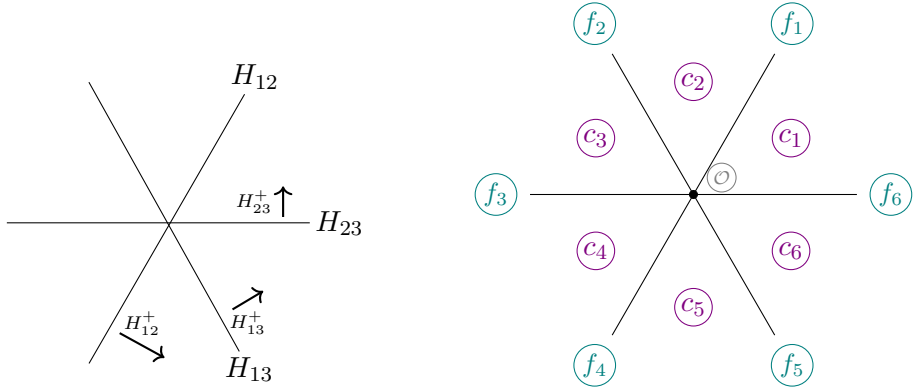
\end{example}

The Tits product on $\F$ was introduced by Tits in \cite{tits1974buildings} and has proved to be a powerful tool in the study of hyperplane arrangements. It makes $\F$ an example of a \emph{left  regular band}, meaning that $FGF = FG$ for every $F,G \in \F$.

The \emph{Tits algebra} $\R \F$ is the semigroup $\R$-algebra over $\F$. Write a typical element in the Tits algebra as $u = \sum_{F \in \F} u_{F} F$, where $u_{F} \in \R$. Similarly, the lattice $\lat$---which has the structure of a semigroup with multiplication given by the join operation $x \vee y$ ---can be turned into a semigroup algebra $\R \lat$. In this way, $\mathfrak{s}: \R \F \to \R \lat$ becomes an algebra homomorphism\footnote{In fact, $\mathfrak{s}$ is the abelianization morphism for $\R \F$.}. 

For any $F \in \F$ and $C \in \ch$, one has $FC \in \ch$. Hence $\R \ch$ is a
left-ideal of $\R \F$ and thus a left $\R \F$-module, where $u \in \R \F$ acts on $C \in \ch$ by
\[u \cdot C = \sum_{F \in \F} u_{F} FC. \]

In their celebrated result, Bidigare, Hanlon and Rockmore provide a way to analyze the action of $u \in \R \F$ on $\R \ch$.

\begin{theorem}
{\rm (Bidigare-Hanlon-Rockmore, \cite[Thm. 1.2]{BHR}).}
\label{thm:bhr}
Let $u = \sum_{F \in \F} u_{F} F \in \F$ act on $\R \ch$. Then for every $X \in \lat$, $u$ has an eigenvalue 
\[\lambda_{X} = \sum_{F \subset X} u_{F} \]
that has multiplicity $|\mu(V,X)|$.
Moreover, if $u_{F} \in \R_{\geq 0}$, then $u$ acts semisimply on $\R \ch$.
\end{theorem}
\begin{remark}
Theorem \ref{thm:bhr} has been applied extensively in the study of random walks on the chambers of hyperplane arrangements. To construct such a random walk, assign coefficients $u_{F}$ to each $F \in \F$ such that $u_{F} \geq 0$ and $\sum_{F \in \F} u_{F} =1$. See \cite{BHR}, \cite{RSW} for details.
\end{remark}

\subsection{Reflection arrangements}\label{sec:reflarr}
The arrangements we are interested in are those coming from reflection groups, called \emph{reflection arrangements.} Given a reflection group $W$, identify $W$ with its canonical faithful representation in $\GL(V)$ for $V = \R^{r}$. The reflection arrangement $\A(W)$ (written $\A$ when the context is clear) is the hyperplane arrangement $\{ H_{i} \}_{i \in I}$ where for each $H_{i}$, there is an element $t_{i} \in W$ that orthogonally reflects over $H$.

There are many deep connections between the properties of $W$ and $\A(W)$. For example, the \emph{characteristic polynomial} of $\A$, defined for a general arrangement by
\[ \chi(\A):= \sum_{X \in \lat(\A)} \mu(V,X) t^{\dim(X)}, \]
can be factored as $\chi(\A) = (t-e_{1})(t-e_{2}) \dots (t-e_{r})$ when $\A$ is a reflection arrangement (where the $e_{j} := d_{j} -1$ are the exponents of $W$).  

By construction, the reflections in $W$ preserve $\A$, inducing an action on $\lat$ and $\F$. This action extends to $\R \F$ and commutes with the support map $\mathfrak{s}$, so $w\cdot \mathfrak{s}(F) = \mathfrak{s}(wF)$. 

The action of $W$ on $\ch$ is simply transitive; because of this, once one makes a choice of a \emph{fundamental chamber} $c_{1} \in \ch$ (which forms a fundamental domain for the $W$-action on $V$), each remaining chamber can be uniquely identified with a $w \in W$ by $c_{w}:= w(c_{1})$. From this it follows that $\R \ch$ is isomorphic as a (left-)$\R W$-module to the group algebra $\R W$ itself.

The localization of a reflection arrangement $\A_{X}$ is always a reflection arrangement with corresponding reflection group $W_{X}$, the point-wise stabilizer of $X$. By contrast, the restriction of a reflection arrangement $\A^{X}$ is not necessarily a reflection arrangement. In fact, using a result of Abramenko \cite[Prop. 5]{abramenko1994walls}, Aguiar and Mahajan show that an essential reflection arrangement $\A$ has the property that $\A^{X}$ is a reflection arrangement for every flat $X \in \lat(\A)$ if and only if $W$ is a direct product of coincidental reflection groups. This fact will prove instrumental in studying the Eulerian representations of coincidental reflection groups.

\subsubsection{Solomon's descent algebra}\label{sec:desalg}
While the results discussed in the Introduction take place in the group algebra, thus far, we have only discussed $\R \F$. Solomon's descent algebra and a theorem of Bidigare provide the tools to translate
between $\R \F$ and $\R W$.

In \cite{solomon} Solomon observed that the descent set of an element of a Coxeter group, as defined in \eqref{descent-set-definition}, could be used to define a subalgebra of $\R W$. In particular, he showed that there is a subalgebra of $\R W$, now known as \emph{Solomon's descent algebra}, defined by
\[ \D(W) = \D:= \Big\{  \sum_{w \in W} c_{w} w : c_{w} = c_{w'} \textrm{ if }  \Des(w) = \Des(w')  \Big\}.\]
What is surprising about Solomon's result is that elements of $\D$ are closed under multiplication. Let 
\[ Y_{T} := \sum_{w: \Des(w) \subseteq T} w, \]
\[ Z_{T} := \sum_{w: \Des(w) = T} w, \]
where $T$ varies over subsets of $S$; the collection of $\{ y_{T} \}_{T \subset S}$ and $\{ z_{T} \}_{T \subset S}$ form bases of $\D$. Elements in $\D$ act by left multiplication on $\R W$. The opposite descent algebra $\dopp$, can be defined as the same set as $\D$ acting by right multiplication on $\R W$. Note that $Y_{T}$ and $Z_{T}$ are also bases for $\dopp$. 

Every reflection arrangement is linked to $\dopp$ as follows. Let $\R \F^{W}$ be the algebra of $W$-invariants of $\R \F$. A basis for $\R \F^{W}$ is indexed by $W$-orbits of elements of $\F$. Define 
\[ \gamma_{[G]} := \sum_{F: [F] \subseteq [G] } F,  \]
\[ \zeta_{[G]} := \sum_{F \in [G]} F. \]
Both $\{ \gamma_{[G]} \}$ and $\{ \zeta_{[G]} \}$ as $[G]$ runs over orbits in $\F$ form bases for $\R \F^{W}$. 

Using the fact that $c_{1} \in \ch$ forms a fundamental domain for the $W$-action on $V$, each face orbit $[G]$ can be identified uniquely with a face $F \subset c_{1}$ where $F \in [G]$. Hence, each $[G]$ can be uniquely identified with every subset $T$ of $S$ by
\[ [G] \iff T:= \{ s \in S: s(F) \neq F  \}\subseteq S. \]
In this case we say that $G$ is of \emph{type $T$}.

From this identification, Bidigare proves a beautiful connection between $\dopp$ and $\R \F^{W}$.
\begin{theorem}{\rm(Bidigare, \cite[Thm. 3.81]{bidigare}).}\label{thm:bidigare} 
There is an algebra isomorphism 
$$
\varphi: \R \F^{W} \to \dopp
$$
given by sending 
\[ \varphi: \gamma_{[G]} \mapsto Z_{T} \]
\[ \varphi: \zeta_{[G]} \mapsto Y_{T},\]
where $[G]$ is of type $T$.
\end{theorem}

\begin{example}
Consider once again the arrangement $\A(S_{3})$ discussed in Example \ref{ex:braid1}. In this case, $W = S_{3}$ and $S = \{ (12), (23) \}$. The face orbits in $\F$ and their face types are:
\begin{align*}
 [\mathcal{O}] = \{ \mathcal{O} \} &\iff \emptyset \\
 [f_{1}] = \{ f_{1}, f_{3}, f_{5} \} &\iff \{ (23) \} \\
 [f_{6}] = \{ f_{2}, f_{4}, f_{6} \} &\iff \{ (12) \} \\
 [c_{1}] = \{ c_{1}, c_{2}, c_{3}, c_{4}, c_{5}, c_{6}\} &\iff \{ (12), (23) \}.
\end{align*}

\end{example}

\subsection{Eulerian idempotents for hyperplane arrangements}\label{sec:eulhyp} In this section we will briefly summarize the most relevant parts of the theory of Eulerian idempotents for hyperplane arrangements initiated by Saliola in \cite{sal2008quiv,saliola2009face,saliola2012eigenvectors} and later studied by Aguiar-Mahajan in \cite{aguiarmahajan}. For generalizations of these idempotents to settings other than hyperplane arrangements, see \cite{berg2011primitive,hivert20111,novelli2010representation,saliolasemigroup}.

In \cite{aguiarmahajan}, Aguiar and Mahajan give a number of equivalent ways to define the idempotents first introduced by Saliola. We will focus on Saliola's original framework of homogeneous sections of the support map (called the \emph{Saliola method} in \cite{aguiarmahajan}).

Let $\mathfrak{u}: \R \lat \to \R \F$ be any section of the support map, meaning that the composition 
\[ \R \lat \overset{\uu}{\longrightarrow} \R \F  \overset{\mathfrak{s}}{\longrightarrow} \R \lat  \]
is the identity on $\R \lat$. For any $X \in \lat(\A)$, define $\uu_{X}:=\uu(X)$. By construction, $\mathfrak{s}(\uu_{X}) = X$ and $\uu_{X}$ can be written as 
\[\uu_{X} = \sum_{F \in \F} u_{F} F,\]
where $\sum_{F \in \F} u_{F} = 1$. 

The section $\uu$ is \emph{homogeneous} if each $\uu_{X}$ is a sum of faces with support exactly $X$:
\[ \uu_{X} = \sum_{\substack{F \in \F \\ \mathfrak{s}(F) = X}} u_{F} F. \]
A homogeneous section $\uu$ is \emph{the uniform section} if $u_{F} = u_{G}$ whenever $\mathfrak{s}(G) = \mathfrak{s}(F)$.

In \cite{saliola2009face}, Saliola shows that using any homogeneous section $\uu$, one can recursively define a family of idempotents $\{ \mathfrak{e}_{X} \}_{X \in \lat}$ in $\R \F$ by
\[ \mathfrak{e}_{X}:= \uu_{X} - \sum_{Y: Y < X} \uu_{X}\cdot \mathfrak{e}_{Y}. \]
He shows that the family $\{ \mathfrak{e}_{X} \}_{X \in \lat}$ (henceforth \emph{flat idempotents}\footnote{In \cite{aguiarmahajan}, the $\mathfrak{e}_{X}$ are called Eulerian idempotents.}) form a complete, primitive and orthogonal system of idempotents in $\R \F$.
Importantly, each family depends on the section being used, and in fact, homogeneous sections are in correspondence with families of flat idempotents $\{ \frake_{X}\}_{X \in \lat}$.

A homogeneous section $\uu$ is an \emph{eigensection} of an element $u \in \R \F$ if there are scalars $\lambda = \{ \lambda_{X} \}_{X \in \lat}$ such that for every $X \in \lat$
\begin{equation}
\label{eigensection-definition}
u^{X} \cdot \uu_{X} = \lambda_{X} \cdot \uu_{X}, 
\end{equation}
where $u^{X} := \sum_{F \subset X} u_{F} F$. Saliola studies eigensections\footnote{In fact he does this in the more general context of left-regular-bands.} in \cite{saliola2012eigenvectors}, where he shows that $\uu$ is an eigensection of $u$ with eigenvalues $\lambda = \{ \lambda_{X} \}_{X \in \lat}$ as in
\eqref{eigensection-definition}
if and only if $u = \sum_{X \in \lat} \lambda_{X} \mathfrak{e}_{X}$.

When $\A$ is a reflection arrangement, Saliola shows in \cite{sal2008quiv} that one can define a complete, orthogonal system of idempotents inside of $\R \F^{W}$ using an \emph{invariant homogeneous section}, which is a homogeneous section such that $u_{F} = u_{G}$ if $[F] = [G]$. For example, the uniform section is an invariant section. A family of \emph{flat-orbit idempotents} $\{ \frake_{[X]} \}_{[X] \in \lat(\A)/W}$ in $\R \F^{W}$ are obtained from an invariant homogeneous section in an analogous way to the non-invariant case. Equivalently, given an invariant section one can group the $\mathfrak{e}_{X}$ by flat-orbit: 
 \[ \mathfrak{e}_{[X]}:= \sum_{Y \in [X]} \mathfrak{e}_{Y}. \]
The flat-orbit idempotents can be realized in $\R W$ using the isomorphism $\varphi: \R \F^{W} \to \dopp$, and the resulting idempotents $\varphi(\frake_{[X]})$ recover primitive idempotents in $\dopp$ defined earlier by F. Bergeron, N. Bergeron, Howlett and Taylor in \cite{bergeron1992decomposition}.

One can define coarser idempotents 
\[\mathfrak{e}_{k}:=  \sum_{\substack{X \in \lat(\A) \\ \dim(X) = k}} \mathfrak{e}_{X} = \sum_{\substack{[X] \in \lat(\A)/W \\ \dim(X) = k}} \mathfrak{e}_{[X]}.\]
In the case of the uniform section, the family $\mathfrak{e}_{k}$ has particularly nice properties; we will call these idempotents the \emph{Eulerian idempotents}, and henceforth, the notation $\mathfrak{e}_{k}$ will refer only to them. When $W$ is $S_{n}$ or $B_{n}$, applying the map $\varphi$ to the $\frak{e}_{k}$ recovers the idempotents defined by equations \eqref{eq:typeaeul} and \eqref{eq:typebeul}, respectively.


 When $W$ is a coincidental reflection group, Aguiar and Mahajan show the $\mathfrak{e}_{k}$ have an elegant expression.
\begin{theorem}{\rm (Aguiar-Mahajan, \cite[Thm. 12.71]{aguiarmahajan}).} \label{thm:AMcoin} For a coincidental reflection group $W$,
\[ \sum_{k=0}^{r} t^{k} \mathfrak{e}_{k} = \sum_{X \in \lat(\A)} \frac{\chi(\A^{X})}{c^{X}} 
\Bigg(
\sum_{F: \mathfrak{s}(F) = X} F 
\Bigg)
=\sum_{[F] \in \F^{W}} \frac{\chi(\A^{\mathfrak{s}(F)})}{c^{\mathfrak{s}(F)}} \zeta_{[F]} \] 
 where $c^{X}$ is the number of chambers in $\A^{X}$.
\end{theorem}

\section{Topology of subspace arrangements}\label{sec:topology}
Our goal is eventually to give a description of the Eulerian representations in terms of the (co)homology of topological spaces closely related to hyperplane arrangements. We discuss those spaces and their (co)homology below.
\subsection{The (equivariant) Goresky-MacPherson formula}\label{sec:WHandGM} 
We will first consider the more general setting of subspace arrangements. 
A \emph{real subspace arrangement} is a collection of linear subspaces $\mathcal{U} = \{ U_{i} \}_{i \in I}$ of an $\R$-vector space $V$. Note that a hyperplane arrangement is a subspace arrangement where every subspace has codimension one. 

As in the case of hyperplane arrangements, let $\lat(\mathcal{U})$ be the poset of intersection subspaces, ordered by reverse containment. In general, $\lat(\U)$ is a not necessarily a geometric lattice, but in the cases relevant to us it will be. Let $(V, X)$ be the open interval between $V$ and $X$ in $\lat(\U)$. The \emph{order complex} of $(V,X)$, written $\Delta(V,X)$, is the simplicial complex with $k$-dimensional faces corresponding to $k$-chains in $(V,X)$; this simplicial complex has homology $H_{*}(V,X)$. 
When $\lat(\U)$ is a geometric lattice, $\tilde{H}_{i}(V,X) = 0$ unless $i =\codim(X) - 2$. Recall from the Introduction that we defined
$$
\WH_{X}:= \tilde{H}_{\codim(X)-2}(V,X),$$
which is a summand of \emph{Whitney homology}; see Bjorner \cite{bjorner1992homology} or Orlik-Terao \cite{orlikterao} for precise definitions of Whitney homology in general.

Let $\com_{\mathcal{U}}:= V - \mathcal{U}$. If a group $W$ acts on $V$ in a way that preserves $\mathcal{U}$, call $\mathcal{U}$ a $W$-subspace arrangement. In this case, $W$ also acts on $\com_{\mathcal{U}}$, making $H^{*}(\com_{\mathcal{U}})$ a $W$-module. Moreover, for every $X \in \lat(\U)$, the set-wise stabilizer $N_{X}$ acts on the order complex $\Delta(V,X)$ and on its homology $H_*(V,X)$.

The $W$-module structure of $H^{*}(\com_{\U})$ is determined in part by a one-dimensional representation of $N_{X}$ for each $X \in \U$ as follows. Define the space  $X^{\dagger} := \mathbb{S}^{\codim(X)-1} \cap X^{\perp}$ in $V$, where $\mathbb{S}^{\codim(X)-1}$ is the unit sphere within the perp space $X^{\perp}$. Thus $X^{\dagger}$ will have non-trivial reduced homology in degree $\codim(X)-1$ only, in which case the homology will be one-dimensional. The group $N_{X}$ acts on $X^{\dagger}$ and hence on $\tilde{H}_{\codim(X)-1}(X^{\dagger})$; the latter action is determined by whether $N_{X}$ reverses or preserves the orientation of the fundamental class of $X^{\dagger}$.

For example, if $\A$ is a rank $r$ reflection arrangement and $X = 0$, then $X^{\perp}= V$, $N_{0} = W$, and $0^{\dagger} = \mathbb{S}^{r-1}$. Hence $W$ acts as the sign representation on $H_{r-1}( \mathbb{S}^{r-1})$ because reflections are orientation reversing; note that $\sgn(w)= \Det_{V}(w)$, the determinant of $w$ acting on $V$. In fact, for any $X \in \lat(\A)$, by an analogous argument, $H_{\codim(X)-1}(X^{\dagger}) \cong_{N_{X}} \Det_{V/X}$, where $\Det_{V/X}$ is the linear character given by the determinant of $w \in N_{X}$ acting on $V/X$. That $N_{X}$ acts on $V/X$ is clear from the fact that $N_{X}$ stabilizes $X$ set-wise; one can then identify the action on $V/X$ with $X^{\perp}$ because $W$ and all of its subgroups preserve inner-products.

From this, we can describe $H^{*}(\com_{\U})$ as both a vector space and a $W$-module.

\begin{theorem}\label{thm:GM}
 Let $\mathcal{U}$ be a real subspace arrangement in $V$. Then
 \begin{enumerate}
    \item { \rm (Goresky-MacPherson: \cite{goresky1988stratified}, \cite{hu1994homology}, \cite{vasilev1994complements}, \cite{ziegler1991homotopy}). } As a vector-space,
    \[ H^{i}(\com_{\mathcal{U}}) \cong \bigoplus_{X \in \lat(\mathcal{U})} \tilde{H}_{\codim(X) - i -2}(V,X). \]
    \item { \rm (Sundaram-Welker: \cite{sundaramwelker}). }If $\mathcal{U}$ is a $W$-subspace arrangement, then as $W$-modules, 
    \[ H^{i}(\com_{\mathcal{U}}) \cong_{W} \bigoplus_{[X] \in \lat(\mathcal{U})/W} \ind_{N_{X}}^{W} \Big( \tilde{H}_{\codim(X) - i -2}(V,X) \otimes \tilde{H}_{\codim(X)-1}(X^{\dagger}) \Big). \]
\end{enumerate}
\end{theorem}

\begin{example}\label{ex:oddcom}
Let $\A$ be a hyperplane arrangement in a real vector space $V$, and consider the subspace arrangement $\A^{d}$ in $V \otimes \R^{d}$ defined by tensoring every $H \in \A$ by $\R^{d}$. Then $\lat(\A)$ is isomorphic to $\lat(\A^{d})$, and both are geometric lattices. When $d =2$ one has $V \otimes \R^{2} \cong \C$ as a real vector space and $\A^{d}$ can be thought of as the complexification of $\A$.

For $d > 1$, let $\oddcom_{\A} = V \otimes \R^{d} - \A^{d}$. Observe that given $X \in \lat(\A)$, 
\[ \codim_{V \otimes \R^{d}}(X \otimes \R^{d}) = d\cdot \codim_{V}(X) \] and $\Delta(V,X) \cong \Delta(V \otimes \R^{d}, X \otimes \R^{d})$. Hence Theorem \ref{thm:GM} implies that
\begin{equation}\label{eqcom}
    H^{i}(\oddcom_{\A}) = \bigoplus_{X \in \lat(\A)}\tilde{H}_{d\cdot \codim_{V}(X) - i - 2}(V,X). 
\end{equation} 
Since $H_{k}(V,X) = 0$ unless $k = \codim_{V}(X) - 2$, the right-hand-side of \eqref{eqcom} is 0 unless $d \cdot \codim_{V}(X)- i -2 = \codim_{V}(X) - 2$, forcing $i = (d-1)\cdot \codim_{V}(X)$.
It follows that as a vector-space, every nonzero component of $H^{*}(\oddcom_{\A})$ is of the form
\[ H^{j(d-1)}(\oddcom_{\A}) = \bigoplus_{\substack{X \in \lat(\A)\\ \codim(X) = j}}\tilde{H}_{j - 2}(V,X) = \bigoplus_{\substack{X \in \lat(\A)\\ \codim(X) = j}} \WH_{X}. \]

To equivariantly describe $H^{*}(\oddcom_{\A})$, note that $N_{X}$  stabilizes $X \otimes \R^{d}$
set-wise for any $X \in \lat(\A)$. Since $w \in N_{X}$ acts on $X^{\dagger}$ by $\Det_{V/X}(w)$, it follows that $w$ acts on $(X \otimes \R^{d})^{\dagger}$ by $(\Det_{V/X}(w))^{d}$. If $d$ is even, this is the trivial representation, while if $d$ is odd, this is $\Det_{V/X}(w)$. Hence:

\[H^{j(d-1)}(\oddcom_{\A}) \cong_{W} \begin{cases} 
\bigoplus_{\substack{[X] \in \lat(\A)/W\\ \codim(X) = j}} \ind_{N_{X}}^{W} \WH_{X} & d \textrm{ is even,}\\

\bigoplus_{\substack{[X] \in \lat(\A)/W\\ \codim(X) = j}} \ind_{N_{X}}^{W} \WH_{X} \otimes \Det_{V/X} & d \textrm{ is odd.}
\end{cases} \]
Recall that $\WH_{[X]}:= \ind_{N_{X}}^{W} \WH_{X} \otimes \Det_{V/X}$. Hence when $d$ is odd, the latter case can be written as 
\[H^{j(d-1)}(\oddcom_{\A})= \bigoplus_{\substack{[X] \in \lat(\A)/W\\ \codim(X) = j}} \WH_{[X]} .\]

When $d = 2$, $H^{*}(\oddcom_{\A})$ is isomorphic (as a graded ring) to the Orlik-Solomon algebra (see \cite{orlikterao}). When $d \geq 3$ is odd, $H^{*}(\oddcom_{\A})$ is isomorphic (as a graded ring) to the
associated graded of the Varchenko-Gelfand ring, to be discussed and defined in Section \ref{sec:vargel}.
\end{example}
\subsection{Equivariant BHR Theory}
When $\A$ is a reflection arrangement, the $W$-module structure of $WH_{[X]}$ can also be framed in terms of the eigenspaces of semisimple operators $u \in \F^{W}$ on $\R \ch$. Given $u \in \F^{W}$ that acts semisimply on $\R \ch$, let $(\R \ch)_{\lambda}$ be the eigenspace of $u$ corresponding to the eigenvalue $\lambda$. For a flat $X \in \lat$, denote by $\lambda_{X}$ the eigenvalue corresponding to $X$ given by Theorem \ref{thm:bhr}.

Reiner, Saliola and Welker give an equivariant formulation of Theorem \ref{thm:bhr}: 
 
\begin{theorem}{ \rm (Reiner-Saliola-Welker \cite[Thm. 4.9]{RSW}).}\label{thm:rsw}
Let $u \in \F^{W}$ act semisimply on $\R \ch$. Then there is an isomorphism of $W$-modules
\[ (\R \ch)_{\lambda} \cong_{W}  \bigoplus_{\substack{[X]: \lambda_{X} = \lambda }} \WH_{[X]}. \]
\end{theorem}
To recover Theorem \ref{thm:bhr} from Theorem \ref{thm:rsw}, note that the dimension of $\WH_{X}$ is $|\mu(V,X)|$.

Our goal in Sections \ref{sec:coincidental} and \ref{sec:allcox} will be to relate the spaces in Theorems \ref{thm:GM} and \ref{thm:rsw} to the Eulerian representations.

\subsection{The Varchenko-Gelfand Ring}\label{sec:vargel}
Equipped with a description of $H^{*}(\com_{\U})$ for any subspace arrangement, we now wish to study in detail two particular (and related) cases: first, the case that the subspace arrangement is a reflection arrangement $\A$ in $V$, and second, the case of $\A^{d} = \A \otimes \R^{d}$ described in Example \ref{ex:oddcom} when $d \geq 3$ and odd. It will turn out that we can describe $H^{0}(\com_{A})$ in terms of $H^{*}(\oddcom_{A})$ and give an explicit construction of the cohomologically graded pieces of the latter. 

For an arrangement $\A = \{H_{i} \}_{i \in I}$, the ring of locally constant functions on $\com_{\A}$ is precisely $H^{0}(\com_{\A})$, and has a filtration by \emph{Heaviside functions}, where for each $i \in I$, the Heaviside function $x_{i} \in H^{0}(\com_{\A})$ is given by
\[ x_{i}(v):= \begin{cases} 1 & v \in H_{i}^{+} \\
0 & v \not \in H_{i}^{+}.
\end{cases}\]

In \cite{varchenkogelfand}, Varchenko and Gelfand use these Heaviside functions to describe $H^{0}(\com_{\A})$. While they give a presentation for non-central arrangements, we will only present the case that $\A$ is central here. To do so, we recall some basic facts about hyperplanes.

Define $E(\A)=E:= \R[e_{i}]_{H_{i} \in \A}$, and for a $k$-tuple of hyperplanes $M = (H_{1}, \dots, H_{k})$, write $e_{M} = e_{1}e_{2} \dots e_{k}$. The set $M$ is \emph{dependent} if 
\[ \codim_{V} \Big( \bigcap_{H_{i} \in M} H_{i} \Big) < |M|  \]
and \emph{independent} otherwise. 
If $M$ is minimally dependent---that is for any $H_{j} \in M$, $M \backslash H_{j}$ is independent---$M$ is called a \emph{circuit}. 
Let $C$ be any circuit of $\A$; then $C$ can be uniquely partitioned into two sets, $C^{+}$ and $C^{-}$ such that 
\[ \bigcap_{H_{i} \in C^{+}} H_{i}^{+} \cap \bigcap_{H_{j} \in C^{-}} H_{j}^{-} = \emptyset.\]

\begin{theorem}{\rm (Varchenko-Gelfand \cite[Thm. 4.5]{varchenkogelfand}).} \label{thm:VG}
The ring morphism defined by
$$
\begin{array}{rcl}
\Psi: E
& \longrightarrow & H^{0}(\com_{\A}) \\
e_{i} &\longmapsto & x_{i}
\end{array}
$$
induces a ring isomorphism $H^{0}(\com_{\A}) \cong E/ \mathcal{J}$, 
with $\mathcal{J} = \ker(\Psi)$ generated by:
\begin{enumerate}
    \item $e_{i}^{2} - e_{i}$ for $H_{i} \in \A$,
    \item For every circuit $C$ in $\A$, 
    \[  \prod_{H_{i} \in C^{+}} e_{i} \prod_{H_{j} \in C^{-}} (e_{j}-1) - \prod_{H_{i} \in C^{+}} (e_{i} -1) \prod_{H_{j} \in C^{-}} e_{j}. \]
\end{enumerate}
\end{theorem}
The map $\Psi$ imposes an ascending filtration on $H^{0}(\com_{\A})$ obtained from the natural degree grading on $E$: the $d^{th}$ layer in the filtration is the span of monomials in the variables $x_i$ having degree at most $d$. We will call its associated graded ring the \emph{associated graded Varchenko-Gelfand ring}, which Varchenko and Gelfand show has the following presentation:
\begin{definition}[Associated graded Varchenko-Gelfand ring]\label{def:vg}
For a central hyperplane arrangement $\A$, let $\VG(\A) = \VG := E / \mathcal{I}(\A)$ be the \emph{associated graded Varchenko-Gelfand ring}, where $\mathcal{I}(\A) = \I$ is generated by:
\begin{enumerate}
    \item $e_{i}^{2}$ for each $H_{i} \in \A$; 
    \item  For every circuit $C$ in $\A$ 
  \[ \tpart(e_{C}) :=  \sum_{H_{i} \in C} c(i) e_{C \backslash H_{i}}, \]
 where 
 \[ c(i) = \begin{cases} 1 & \textrm{if }H_{i} \in C^{-}, \\
 -1 & \textrm{if } H_{i} \in C^{+}.
 \end{cases} \]
\end{enumerate}
Let $\VG^{k}$ be the $k$-th graded piece of $\VG$ spanned by degree $k$ polynomials in the $e_{i}$.
\end{definition}

 \begin{remark}
The relations in $\VG$ bear a striking resemblance to the graded Orlik-Solomon algebra (see \cite{orlikterao}). However, the former ring is commutative while the latter is anti-commutative. For more discussion on this distinction, see Moseley \cite{moseley2017equivariant}, or Moseley-Proudfoot-Young \cite{moseley-proudfoot-young}. The ring $\VG$ is also easily confused with another graded commutative ring $\mathbb{U}(\A)$ defined by Orlik-Terao in \cite{orlik1994commutative}. However $\mathbb{U}(\A)$ is not isomorphic to $\VG$, as $\VG$ depends only on the underlying oriented matroid while $\mathbb{U}(\A)$ depends on the coordinates of $\A$. Cordovil carefully discusses this distinction in \cite{cordovil2002commutative}.
 \end{remark}
 
Varchenko-Gelfand show in \cite{varchenkogelfand} that when $\A$ is central, $\VG$ has an \emph{nbc-basis}; see Cordovil \cite[Cor. 2.8]{cordovil2002commutative}. To define an nbc-basis, impose an ordering on the hyperplanes in $\A$. A \emph{broken circuit} is a circuit with its largest element removed, and an \emph{nbc-set (non-broken-circuit-set)} is a set of hyperplanes that does not contain a broken circuit. The monomials indexed by nbc-sets of size $k$ form a basis for $\VG^{k}$.

\begin{example}[Braid Arrangement]\label{ex:VGbraid} Once again, let $W = S_{n}$, so $\A(S_{n})$ is the Braid arrangement with hyperplanes of the form $H_{ij}:= \{ x_{i} - x_{j} = 0 \}$. 
There is a nice description of the nbc-basis in this case (see Barcelo-Goupil in \cite{barcelo1995non}) as the monomials formed by choosing one element from each of the $n-1$ sets:
\[ \{ 1, e_{12} \}, \{ 1, e_{13}, e_{23} \}, \dots, \{ 1, e_{1n}, \dots ,e_{(n-1)n}\}. \]

Using the nbc-basis, it can be shown that the only circuits needed to generate $\I$ are of the form $C^{+} = \{ H_{ij}, H_{jk} \} $, $C^{-} = \{ H_{ik} \}$. 
Hence $\mathcal{I}$ is generated by 
\begin{enumerate}
    \item $e_{ij}^{2}$ for every $H_{ij} \in \A$
    \item $e_{ij}e_{jk} - e_{ij}e_{ik} - e_{jk}e_{ik}$ for every $H_{ij}, H_{jk},H_{ik} \in \A$.
\end{enumerate}
\end{example}
\begin{example}\label{ex:VGtypeB}(Type $B$) 
Let $W = B_{n}$. There are three types of hyperplanes in $\A(B_{n})$: they are, for $1 \leq i < j \leq n$, 
\[ H_{ij} = \{ x_{i} - x_{j} = 0 \}, \hspace{2em} \overline{H}_{ij} = \{ x_{i} + x_{j}  = 0 \}, \hspace{2em} H_{i} = \{ x_{i} = 0 \}. \]
Let the corresponding generators in $\VG$ be $e_{ij}$, $\Oe_{ij}$ and $e_{i}$. The nbc-monomials in this case are obtained by multiplying one element from each of the $n$ sets:
\[ \{ 1, e_{1} \}, \{ 1, e_{12}, \Oe_{12}, e_{2} \}, \dots, \{ 1, e_{1n}, \dots, e_{(n-1)n}, \Oe_{1n}, \dots, \Oe_{(n-1)n}, e_{n} \}. \]
In \cite{xico} Section 3.6, Xicotencatl shows that $\I$ is generated by seven quadratic relations in the $e_{ij}$, $\Oe_{ij}$ and $e_{i}$. In his notation, $H^{*}F_{\Z_{2}}(\R^{n}- \{ 0 \}, q)$ gives the Varchenko Gelfand ring for $\A(B_{q})$ if $n$ is odd, and his generators $A_{ji}$, $\overline{A}_{ji}$ and $A_{i0}$ correspond to our $e_{ij}$, $\Oe_{ij}$ and $e_{i}$, respectively. 
\end{example}

\begin{remark}
Both the type $A$ and $B$ reflection arrangements are \emph{supersolveable}, meaning that $\lat(\A)$ is a supersolveable lattice for both arrangements; see \cite{stanley2017supersolvable} for an in-depth treatment of supersolveable lattices. It is for precisely this reason that we have the particularly nice choice of bases described in Examples \ref{ex:VGbraid} and \ref{ex:VGtypeB}. In fact, the only real coincidental group that does \emph{not} have a supersolveable arrangement is $H_{3}$. This may be one reason the coincidental groups are named as such---they are types $A$, $B$, the dihedral group $I_{2}(m)$, and \emph{by coincidence} $H_{3}$. Thank you to the anonymous referee for pointing this out.
\end{remark}
 \begin{remark}
 When a group $W$ acts on $\A$ (and therefore also on $\com(\A)$), one naturally obtains a $W$-action on the Varchenko-Gelfand ring and its associated graded $\VG(\A)$ as follows. For $H_{i} \in \A$, let $\alpha_{i} \in H_{i}^{+}$ be the unit length normal vector to $H_{i}$. By construction, for $w \in W$,
 \[ w \cdot \alpha_{i} = c_{j} \alpha_{j}  \]
 where $c_{j} \in \{ \pm 1\}$, $H_{j} \in \A$ and $\alpha_{j} \in H_{j}^{+}$ is the unit length normal vector to $H_{j}$.  
 
 This induces an action on the Heaviside function $x_{i}$:
 \[ w \cdot x_{i} = \begin{cases} x_{j} & c_{j} = 1\\
 1-x_{j} & c_{j} = -1.
 \end{cases} \]
 The identification of $x_{i}$ with $e_{i}$ in
 Theorem \ref{thm:VG} then describes the $W$ action on the Varchenko-Gelfand ring. 
 To obtain a $W$-action on $\VG(\A)$, we again filter by degree so that
 \[ w \cdot e_{i} = c_{j}e_{j}.\]
 For example, in the braid arrangement, if one considers the hyperplane $H_{ij}$ and transposition $(ij)$, then
\begin{align*}
    (ij) \cdot \alpha_{ij} &= -\alpha_{ij} \text{ in } \com(\A),\\
    (ij) \cdot x_{ij} &= 1 - x_{ij} \text{ in } H^{0} \com(\A), \text{ and}\\
    (ij) \cdot e_{ij} &= -e_{ij} \text{ in } \VG(\A).
\end{align*}
 \end{remark}

There is an connection between $\VG$ and the cohomology ring of $\oddcom_{\A}$ for odd $d$ due to Moseley, which will prove instrumental to our results in Sections \ref{sec:coincidental} and \ref{sec:allcox}. 
\begin{theorem}{ \rm(Moseley, \cite[Thm. 1.4]{moseley2017equivariant}).} \label{thm:moseley} Let $\A$ be a real hyperplane arrangement and $d \geq 3$ an odd integer. Then $\VG$ is isomorphic as a graded ring to $H^{*}(\oddcom_{\A})$, with 
\[ \VG^{k} \cong H^{k(d-1)}(\oddcom_{\A}). \] 
If a finite group $W$ acts on $\A$, this isomorphism is $W$-equivariant. 
\end{theorem}
In the case that $\A$ is an irreducible reflection arrangement with reflection group $W$, Theorem \ref{thm:moseley} says that 
\[ \VG(\A) \cong_{W} H^{*}(\oddcom_{\A}). \] If we then consider the localization arrangement $\A_{X}$ for $X \in \lat(\A)$, one has $N_{X}$, the set-wise stabilizer of $X$, acting on $\A_{X}$. Theorem \ref{thm:moseley} then implies that 
\[ \VG(\A_{X}) \cong_{N_{X}} H^{*}(\oddcom_{\A_{X}}). \]
\section{Eulerian representations for coincidental reflection groups}\label{sec:coincidental}
In this section, we will draw upon the theories in Sections \ref{sec:hyperplanes} and \ref{sec:topology} to develop a unified theory of Eulerian representations for coincidental reflection groups. 

Given a coincidental reflection group $W$, let $S$ be the Coxeter generators of $W$, $\A$ be its reflection arrangement, and $r$ be the rank of $\A$ (or equivalently, $|S|$). Because $W$ is coincidental, the exponents (equivalently, degrees) of $W$ can be expressed in terms of the {\it exponent gap} $g$ as
$$
1, 1+g, 1+2g, \dots, 1+ (r-1)g.
$$
Here are the ranks $r$ and exponent gaps $g$ for the coincidental groups:
\begin{center}
\begin{tabular}{|c|c|c|}\hline
     $W$ & $r$ & $g$  \\ \hline\hline
     $S_n$ &$n-1$ & $1$ \\ \hline
     $B_n$ &$n$ & $2$ \\ \hline
     $H_3$ &$3$ & $4$ \\ \hline
    $I_2(m)$ &$2$ & $m-2$ \\ \hline
\end{tabular}
\end{center}

\subsection{Generalizing Barr's element}
The key ingredient in developing the theory of the Eulerian representations for coincidental reflection groups lies in generalizing the technique used in Type $A$ by Barr in \cite{barr} and Gerstenhaber-Schack in \cite{Gerstenhaber-Schack}.

\begin{definition} For any reflection group $W$, the \emph{Barr-element} in $\R W$ is
\[ \s(W) := \sum_{s \in S} \sum_{\substack{w \in W \\ \Des(w) \subset \{ s\}}} w. \]
\end{definition}
Equivalently, in the notation of Section \ref{sec:desalg}, 
$$
\s(W) = \sum_{s \in S} Y_{\{s \}}.
$$
As before, when the context is clear write $\s(W) = \s$. Recall that $\varphi: \R \F^{W} \to \dopp$ is the algebra isomorphism in Theorem \ref{thm:bidigare}.
Let $\tilde{\s}:= \varphi^{-1}(\s).$
By Theorem \ref{thm:bidigare}, it follows that 
\[ 
\tilde{\s} = \sum_{\substack{[G] \in \F^{W} \\ \dim(G) = 1}} \zeta_{[G]} = \sum_{\substack{F \in \F\\\dim(F) = 1}}F \in \R \F^{W}. 
\]

\begin{prop}\label{prop:barrprop}
For any reflection group $W$, the element $\s$ acts semisimply on $\R W$. When $W$ is coincidental, $\s$ has $r+1$ eigenvalues $\sigma_{0} < \sigma_{1} < \dots < \sigma_{r}$ such that $\sigma_{k}$ is the number of 1-dimensional faces in $\A^{X}$ for any flat $X$ of dimension $k$. 
\end{prop}
\begin{proof}
Because $\tilde{\s}$ has non-negative coefficients, it acts semisimply on $\R \ch$ by Theorem \ref{thm:bhr}, and thus $\s$ acts semisimply on $\R W$ by the algebra isomorphism in Theorem \ref{thm:bidigare}.  

Suppose $W$ is coincidental. Then for any $X,Y \in \lat$, one has $\A^{X} \cong \A^{Y}$ if and only if $\dim(X) = \dim(Y)$. By Theorem \ref{thm:bhr} the eigenvalues of $\s$ are indexed by $\sigma_{X}$ for each $X \in \lat$  and 
\[\sigma_{X} = \sum_{\substack{F \subset X \\ \dim(F) = 1}}1 = \# \{ F \subset X: \dim(F) = 1 \}. \]
A flat $X$ is spanned by the 1-dimensional faces it contains, and therefore if $X > Y $ in $\lat(\A)$ then $\sigma_{X} < \sigma_{Y}$. Thus $\sigma_{X}$ depends only on $\dim(X)$ and so can be written $\sigma_{k}$ for $0 \leq k \leq r$. 
\end{proof}
\begin{remark}
Proposition \ref{prop:barrprop} gives another explanation for why (as Barr observed) the eigenvalues of $\s(S_{n})$ are $\{2^{k+1} -2 \}$ for $0 \leq k \leq n-1$. In Type $A$ a dimension $k$ flat $X$ has restriction arrangement $\A^{X}(S_{n}) \cong \A(S_{k+1})$. Hence $\sigma_{k}$ counts the number of one-dimensional faces in $\A(S_{k+1})$, which correspond to ordered pairs of non-empty, disjoint subsets $(I,J)$ with $I \sqcup J = \{ 1,2,\dots, k+1\}$ (see Aguiar-Mahajan \cite[Sec. 6.3]{aguiarmahajan}). Given $(I,J)$, the face $F$ is where the coordinates in $V$ have $x_{i}$ constant for all $i \in I$ and $x_{j}$ constant for all $j \in J$, and where $x_{i}<x_{j}$ for $i \in I, j\in J$. There are $2^{k+1}-2$ choices for the set $\emptyset \subsetneq I \subsetneq \{1,\dots, k+1 \}$, determining the pair $(I,J)$.

Similarly, we can deduce from Proposition \ref{prop:barrprop} that in Type $B$, the eigenvalues of $\s(B_{n})$ are $\{ 3^{k} - 1 \}$ for $0 \leq k \leq n$ using the fact that a dimension $k$ flat $X$ has restriction arrangement $\A^{X}(B_{n}) \cong \A(B_{k})$. One-dimensional faces in $\A(B_{k})$ correspond to the $3^k-1$ assignments 
\[ \epsilon: \{ 1,\dots, k\} \longrightarrow \{+1,-1,0 \}\]
that avoid being identically zero (see Aguiar-Mahajan \cite[Sec. 6.7]{aguiarmahajan}). Given such an assignment $\epsilon$, the face $F$ is where $x_{j}=0$ if $\epsilon(j)=0$ and $\epsilon(i) x_i =\epsilon(j) x_j >0$  for $\epsilon(i),\epsilon(j) \neq 0.$ 

The eigenvalues for $\s(I_{2}(m))$ and $\s(H_{3})$ are listed below.
\begin{center}
\begin{tabular}{|c|c|c|c|c|}\hline
      $W$ &  $\sigma_{0}$ &  $\sigma_{1}$ &  $\sigma_{2}$ &  $\sigma_{3}$  \\ \hline\hline
    $I_{2}(m)$ & 0 & 2 & $2m$ &  \\ \hline
    $H_{3} $ &0 & 2 & 12 & $62$ \\ \hline
\end{tabular}
\end{center}
\end{remark}
Because $\s$ acts semisimply on $\R W$, by Lagrange interpolation, there are idempotents that project onto each eigenspace of $\s$. Let $(\R W)_{\sigma_{k}}$ be the $\sigma_{k}$-eigenspace of $\s$ and denote by $\s^{(k)}$ the projector onto $(\R W)_{\sigma_{k}}$:
\[ \s^{(k)} :=  \prod_{j \neq k} \frac{\s - \sigma_{j}}{\sigma_{k} - \sigma_{j}}.  \]
By construction, 
\[ \s = \sum_{k = 0}^{r} \sigma_{k} \cdot \s^{(k)}.\]
\begin{example}
In $\A(S_{3})$, one has $\s(S_{3}) = 2 + (12) + (23) + (123) + (132)$ and
\begin{align*}
& \s^{(0)} = \frac{1}{6} \big( 2 - (12) - (23) - (123) - (132) + 2(13) \big), \\
& \s^{(1)} = \frac{1}{2} \big( 1 - (13) \big), \\
& \s^{(2)} = \frac{1}{6} \big(1 + (12) + (23) + (13) + (123) + (132) \big). \\
\end{align*}
\end{example}

We will see in Theorem \ref{thm:coincidental} that $\s^{(k)}$ in Types $A$ and $B$ are precisely the Type $A$ and $B$ Eulerian idempotents from equations \eqref{eq:typeaeul} and \eqref{eq:typebeul}, and more generally:
\begin{prop}\label{prop:eulmatch}
When $W$ is coincidental,
$\s^{(k)} = \varphi(\mathfrak{e}_{k})$ for $0 \leq k \leq r$.
\end{prop}
\begin{proof}
Using the notation in Section \ref{sec:eulhyp},
\[ \tilde{\s}^{X} = \sum_{\substack{F \subset X \\ \dim(F) = 1}} F, \]
which is invariant under the action of the reflection group corresponding to $\A^{X}$ for any $X$. By an analogous argument to Aguiar-Mahajan \cite[Lemma 12.70]{aguiarmahajan} this implies that the uniform section is an eigensection for $\tilde{\s}$. Moreover, $\tilde{\s}$ is \emph{separating}, meaning that if $X > Y$ in $\lat(\A)$ then $\sigma_{X} < \sigma_{Y}$. By another theorem of Aguiar and Mahajan \cite[Thm. 12.17]{aguiarmahajan}, it follows that the uniform section is the unique eigensection for $\tilde{\s}$, giving the family of flat idempotents $\{ e_{X} \}_{X \in \lat}$ such that
$\tilde{\s} = \sum_{X \in \lat} \sigma_{X} \mathfrak{e}_{X}$. Because $\sigma_{X} = \sigma_{Y}$ when $\dim(X) = \dim(Y)$,
\[ \tilde{\s} = \sum_{k=0}^{r} \sigma_{k} \cdot \mathfrak{e}_{k}.\]
Now apply $\varphi.$
\end{proof}

\subsection{Implications for Eulerian subalgebras}
Recall that in the Introduction, we defined
an $\R$-linear subspace of $\R W$ called the \emph{Eulerian subspace},  
\[ \E(W):= \E = \Big\{  \sum_{w \in W} c_{w} w : c_{w} = c_{w'} \textrm{ if }  \des(w) = \des(w')  \Big\}.\]
While $\E$ always exists as a subspace of $\R W$ for any reflection group $W$, it is natural to ask whether $\E$ forms a subalgebra (i.e. like $\D$, is closed under multiplication). In Types $A$ and $B$, this is known to be true (see \cite{garsia}, \cite{Bergeron-Bergeron}); our framework allows us to answer this question for any reflection group.  

\begin{theorem}\label{thm:eulsubalg}
The Eulerian subspace $\E$ is a subalgebra if and only if $W$ is coincidental. Moreover, when the Eulerian subalgebra exists, it is always commutative. 
\end{theorem}
\begin{proof}
The Eulerian subspace $\E$ always has dimension $r+1$ by definition. If $\E$ is a subalgebra, it will contain the subalgebra $\R \s$ generated by the Barr element $\s$.  Since $\s$
always acts semisimply, $\R \s$ will be commutative and have dimension equal to the number of distinct eigenvalues of $\s$.
Recall from the proof of Proposition~\ref{prop:barrprop}
that these eigenvalues are indexed by flats $X$ in $\lat(\A)$, with eigenvalue $\sigma_X$ equal to the number of of 1-dimensional faces (henceforth \emph{rays}) in $\A^{X}$.  Since
a flat $X$ is spanned by the rays it contains,
any complete flag of flats $\{0\}=X_0 \subsetneq X_1 \subsetneq \cdots \subsetneq X_{r-1} \subsetneq X_r=V$ gives rise to at least
$r+1$ distinct eigenvalues
$\sigma_{X_0} < \cdots < \sigma_{X_r}$.  Hence
$\E$ is a subalgebra if and only if $\E=\R \s$
if and only if any two flats $X, Y$ of the same dimension have $\sigma_X=\sigma_Y$. Proposition~\ref{prop:barrprop} showed that this occurs whenever $W$ is coincidental.

We check here that when $W$ is not coincidental, one
always has at least two flats $X, Y$ of the same
dimension with $\sigma_X \neq \sigma_Y$. This can be verified computationally for $H_{4}$ and $F_{4}$. The existence of such flats for $E_{6},E_{7}$ and $E_{8}$ can be deduced from computations in Orlik-Terao \cite[Appendix D]{orlikterao}, where they compute the number of lines in all possible arrangements $\A^{X}$ of rank $3$; every line must contain exactly 2 rays.

In Type $D$, a result of Barcelo-Ihrig \cite[Thm. 4.1]{barcelo1999lattices} describes a bijection between flats of $\A(D_{n})$ (for $n \geq 4$) and partitions $\lambda$ of the set $\{ \overline{1}, \dots, \overline{n}, 1, \dots, n\}$ such that
\begin{enumerate}
    \item $\lambda$ has at most one \emph{zero-block} $\lambda_{0}$ where if $i \in \lambda_{0}$ then $\overline{i} \in \lambda_{0}$ and $|\lambda_{0}| \neq 2$; and
    \item every non-zero block $\lambda_{j}$ of $\lambda$ has a partner $\overline{\lambda_{j}}$; if $i \in \lambda_{j}$, then $\overline{i} \in \overline{\lambda_{j}}$ (with $\overline{\overline{i}} = i$). 
\end{enumerate}

Write such a partition as $\lambda = (\overline{\lambda_{k}}, \dots, \overline{\lambda_{1}}, \lambda_{0}, \lambda_{1},\dots, \lambda_{k})$, even if $\lambda_{0} = \emptyset$.
Let $X_{\lambda}\in \lat(\A)$ be the flat corresponding to $\lambda$; when $\lambda$ has $2k$ non-zero blocks, $X_{\lambda}$ has dimension $k$. If $\lambda$ is a refinement of a partition $\lambda^{'}$, then $X_{\lambda} < X_{\lambda^{'}}$ in $\lat(\A)$. Consider the partition $\lambda$ with $\lambda_{0} = \emptyset$, $\lambda_{1} = \{ 1\}$ and $\lambda_{2} = \{ 2,3, \dots, n\}$. 
The partitions refined by $\lambda$ with two non-zero parts (corresponding to lines in $\A^{X_{\lambda}}$) are: 
\begin{enumerate}
    \item  $(\overline{\lambda_{1}} \cup \overline{\lambda_{2}}, \hspace{.3 em} \lambda_{0}, \hspace{.3 em} \lambda_{1} \cup \lambda_{2}) = (\{ \overline{1}, \overline{2}, \dots, \overline{n},\}, \{ \emptyset \}, \{ 1,2, \dots, n\}),$ 
    \item $(\lambda_{1} \cup \overline{\lambda_{2}},\hspace{.3 em} \lambda_{0},\hspace{.3 em} \overline{\lambda_{1}} \cup \lambda_{2}) = (\{ 1, \overline{2}, \dots, \overline{n},\}, \{ \emptyset \}, \{ \overline{1} ,2, \dots, n\}),$ and
    \item $(\overline{\lambda_{1}}, \hspace{.3 em} \overline{\lambda_{2}} \cup \lambda_{0} \cup \lambda_{2}, \hspace{.3 em} \lambda_{1}) = (\{ \overline{1}\}, \{ \overline{2}, 2, \dots, \overline{n}, n\}, \{ 1\})$.
\end{enumerate}
Let $\rho$ be the partition with $\rho_{0} = \emptyset$, $\rho_{1} = \{ 1,2\}$ and $\rho_{2} = \{ 3,\dots, n\}$. The partitions refined by $\rho$ which correspond to lines in $\A^{X_{\rho}}$ are
\begin{enumerate}
    \item   $(\overline{\rho_{1}} \cup \overline{\rho_{2}}, \hspace{.3 em} \rho_{0}, \hspace{.3 em} \rho_{1} \cup \rho_{2}) = (\{ \overline{1}, \overline{2}, \dots, \overline{n},\}, \{ \emptyset \}, \{ 1,2, \dots, n\}),$ 
    \item $(\rho_{1} \cup \overline{\rho_{2}}, \hspace{.3 em} \rho_{0}, \hspace{.3 em} \overline{\rho_{1}} \cup \rho_{2}) = (\{ 1, 2, \overline{3},\dots, \overline{n},\}, \{ \emptyset \}, \{ \overline{1} ,\overline{2}, 3 \dots, n\}),$ 
    \item $(\overline{\rho_{1}}, \hspace{.3 em} \overline{\rho_{2}} \cup \rho_{0} \cup \rho_{2}, \hspace{.3 em} \rho_{1}) = (\{ \overline{1}, \overline{2} \}, \{ \overline{3}, 3, \dots, \overline{n}, n\}, \{ 1,2\})$, and
    \item $(\overline{\rho_{2}}, \hspace{.3 em} \overline{\rho_{1}} \cup \rho_{0} \cup \rho_{1}, \hspace{.3 em} \rho_{2}) = (\{ \overline{3}, \dots, \overline{n} \}, \{ \overline{1}, 1, \overline{2}, 2\}, \{3, \dots, n\})$.
\end{enumerate}
Hence $\A^{X_{\lambda}}$ contains 6 rays and $\A^{X_{\rho}}$ contains 8 rays, but both $X_{\lambda}$ and $X_{\rho}$ have dimension $2$ in $\lat(\A)$.
\end{proof}
As in the case of Solomon's descent algebra $\D$, where we defined bases $Y_{T}$ and $Z_{T}$ for $T \subset S$ (see Section \ref{sec:desalg}), there are natural bases for $\E$ (and $\E^{\mathrm{opp}}$) defined by
 \[y_{\ell} :=\sum_{w: \des(w) \leq \ell} w = \sum_{\substack{T \subset S \\ |T| =\ell }} Y_{T}, \]
\[z_{\ell} :=  \sum_{w: \des(w) = \ell} w = \sum_{\substack{T \subset S \\ |T| =\ell }} Z_{T}\]
for $0 \leq \ell \leq r$. These bases will be useful in defining a generating function for the $\s^{(k)}$ in the subsequent section (Theorem \ref{thm:coincidental} \eqref{7}). 

\subsection{Main results}
At last, we can collect the theories developed in previous sections to give a description of the representations generated by the Eulerian idempotents. Recall 
that
\begin{equation}\label{eq:beta}
     \beta_{W,\ell}:= \frac{1}{|W|} \prod_{i=1}^{\ell} (t-e_{i}) \prod_{i=1}^{r-\ell} (t+e_{i}).
\end{equation}

\begin{theorem}\label{thm:coincidental} Let $W$ be a real coincidental reflection group. Then for $0 \leq k \leq r$, the following are isomorphic as $W$-representations:
\begin{enumerate}
    \item \label{2} $\VG^{k}$, the $k$-th graded piece of the associated graded Varchenko-Gelfand ring;
    \item \label{3} $H^{(d-1)k}(\oddcom_{\A})$ for $d \geq 3$ and odd;
    \item \label{4} $\bigoplus_{[X]} \WH_{[X]}$, where the direct sum is over all $[X] \in \lat(\A)/W$ with $\codim(X) = k$;
    \item \label{5}$(\R W)_{\sigma_{r-k}}$, the $\sigma_{r-k}$-eigenspace of $\s$;
    \item \label{6} $\mathfrak{e}_{r-k} \R \ch$, the representation generated by the Eulerian idempotent $\mathfrak{e}_{r-k}$; 
    \item \label{7} The left $\R W$-module $\R W E_{r-k}$, 
    where $\{E_k\} \subset \E(W)$ are idempotents defined by
    \[ \sum_{k = 0}^{r} t^{k}E_{k} := \sum_{\ell = 0}^{r} \beta_{W,\ell}(t) \cdot z_{\ell}. \]
\end{enumerate}
\end{theorem}
\begin{proof}
\noindent
\begin{itemize}
\item \eqref{2} if and only if \eqref{3}: follows from Theorem \ref{thm:moseley};
\item \eqref{3} if and only if \eqref{4}: follows from Theorem \ref{thm:GM} (2), applied as in Example \ref{ex:oddcom};
\item \eqref{4} if and only if \eqref{5}: follows from Theorem \ref{thm:rsw} and Proposition \ref{prop:barrprop} because a flat $X$ with $\codim(X) = k$ has $\sigma_{X} = \sigma_{r-k}$;
\item 
\eqref{5} if and only if \eqref{6}: follows from Proposition \ref{prop:eulmatch};
\item \eqref{6} if and only if \eqref{7} will follow from translating Theorem \ref{thm:AMcoin} into the context of the polynomial algebra $\dopp[t]$. Specifically, applying $\varphi$ to Theorem \ref{thm:AMcoin} gives 
\begin{equation}\label{eqeul}  \sum_{k=0}^{r} t^{k} \varphi(\mathfrak{e}_{k}) = \sum_{T \subset S} \frac{\chi(\A^{T})}{c^{T}} Y_{T},\end{equation}
where $\chi(\A^{T})$ is the characteristic polynomial of $\A^{\mathfrak{s}(F)}$ for some $F \in \F$ of type $T$, and $c^{T}$ is the number of chambers in $\A^{\mathfrak{s}(F)}$. The right side of \eqref{eqeul} can be simplified by grouping together subsets $T \subset S$ of the same size. Because $W$ is coincidental, $\chi(\A^{T})$ and $c^{T}$ depend only on the size of $T$. In particular, when $|T| = j$, 

\[ \chi(\A^{T}) = (t-1)(t-1-g)\hdots (t-1-g(j-1)) = g^{j} \Big( \frac{t-1}{g}\Big)_{j} \] and 
\[ c^{T} = (2)(2+g) \hdots (2+g(j-1))= g^{j}\Big( \frac{2}{g}\Big)_{j}.\]
Hence
\begin{equation} \label{eqeul2}  \sum_{k=0}^{r} t^{k} \varphi(\mathfrak{e}_{k}) = \sum_{j= 0}^{r} \sum_{\substack{T \subset S \\ |T| = j}}   \frac{\big(\frac{t-1}{g}\big)_{j}}{\big( \frac{2}{g}\big)_{j}} Y_{T} =  \sum_{j= 0}^{r}    \frac{\big(\frac{t-1}{g}\big)_{j}}{\big( \frac{2}{g}\big)_{j}} y_{j} .\end{equation}
Note that an element $w \in W$ with descent set $U$ of size $\ell$ will appear on the right hand side of \eqref{eqeul2} each time $U \subset T$ as $T$ varies over every subset of $S$. The set $U$ will appear in exactly $\binom{r-\ell}{j-\ell}$ subsets of size $j$. Thus
\begin{align}
\label{sum-that-will-use-Chu-Vandermonde}
  \sum_{k=0}^{r} t^{k} \varphi(\mathfrak{e}_{k}) &= \sum_{\ell=0}^{r}  \sum_{\substack{U \subset S \\ |U| = \ell}} Z_{U} \Big( \sum_{j=\ell}^{r} \binom{r-\ell}{j-\ell}\frac{\big(\frac{t-1}{g}\big)_{j}}{\big( \frac{2}{g}\big)_{j}} \Big)  = \sum_{\ell=0}^{r}z_{\ell} \Big( \sum_{j=\ell}^{r} \binom{r-\ell}{j-\ell}\frac{(\frac{t-1}{g})_{j}}{(\frac{2}{g})_{j}}\Big).
\end{align}
One can then check that on the right side of \eqref{sum-that-will-use-Chu-Vandermonde}, the innermost summation appearing in front of $z_{\ell}$ equals $\beta_{W,\ell}(t)$, via
the {\it Chu-Vandermonde summation} formula
\[ {}_{2}F_{1}\left(
\left. \begin{matrix}-n & b\\
& c\end{matrix}
\right| 1\right) = \frac{(c-b)_{n}}{(c)_{n}} \]
with $n=r-\ell$, $b = -(\frac{t-1}{g} - \ell)$ and $c = \frac{2}{g} + \ell$.
\end{itemize}
\end{proof}
\begin{remark}
An equivalent way to write $\beta_{W,\ell}$ is as \begin{equation}\label{eq:originalbeta} \beta_{W,\ell}(t):=\frac{\left( \frac{t+g-1}{g}-\ell \right)_\ell \left(\frac{t+1}{g}\right)_{r-\ell} }{\left(\frac{2}{g}\right)_r},\end{equation}
where $g$ is the exponent gap of $W$ and $(t)_{k} := (t)(t+1) \dots (t+k-1)$ is the rising factorial. We originally used the form of $\beta_{W,\ell}$ given by \eqref{eq:originalbeta}, and thank the anonymous referee for pointing out the form of $\beta_{W,\ell}$ used in \eqref{eq:beta}. To recover \eqref{eq:originalbeta} from \eqref{eq:beta}, note that when $W$ is coincidental, $e_{i} = 1 + (i-1)g$.
\end{remark}
\begin{remark}
Surprisingly, the polynomial $\beta_{W,\ell}$ defined in \eqref{eq:beta} has appeared before in the work of Reiner, Shepler and Sommers \cite[Thm 1.1]{reiner2019invariant}. Let $\bigwedge ^{\ell}V$ be the $\ell$-th exterior power of the reflection representation $V$ of a real coincidental group $W$. Then by Theorem 1.1 in \cite{reiner2019invariant},
\begin{equation}\label{eq:otherbeta}
\beta_{W,\ell}(t)= \frac{1}{|W|} \prod_{i=1}^{\ell}(t - e_{i}) \prod_{i=1}^{r-\ell} (t+ e_{i}) = \frac{1}{|W|} \sum_{w \in W} \frac{\chi_{\bigwedge^{\ell}V}(w)}{\chi_{\bigwedge^{\ell}V}(1)} \cdot t^{\dim(V^{w})},\end{equation}
where $\chi_{\bigwedge^{\ell}V}$ is the character of $\bigwedge^{\ell}V$ and $V^{w}$ is the $w$-fixed space of $V$. 

To derive \eqref{eq:otherbeta} from \cite{reiner2019invariant}, one needs to do a bit of work. In particular, Theorem 1.1 computes the Hilbert series of the $W$-invariant space 
\[ \hilb( (S(V^{*}) \otimes \bigwedge V^{*} \otimes \bigwedge V)^{W}, q,z,s), \] where $V^{*}$ is the dual of $V$ and $S(V^{*})$ is the symmetric algebra of $V^{*}$. We extract the $s^{\ell}$ coefficient of this Hilbert series and set $z = -q^{t}$. Combining \cite[Eq. 2.11]{reiner2019invariant} and taking the limit as $q$ goes to 1 then gives \eqref{eq:otherbeta}. The motivation for this comes from the theory of graded parking functions; see \cite[\S 10]{reiner2019invariant}. Note that Theorem 1.1 is in terms of exponents $e_{i}$ and \emph{co-exponents} $e_{i}^{*}$, but in the case of real reflection groups, $e_{i} = e_{i}^{*}$. 

We thank the anonymous referee for pointing out this remarkable and mysterious connection.
\end{remark}

\begin{remark}
The generating functions for the Type $A$ and $B$ Eulerian idempotents  (equations \eqref{eq:typeaeul}, \eqref{eq:typebeul}) are easily deduced from Theorem \ref{thm:coincidental} \eqref{7} by taking $r =n-1,g =1$ in Type $A$ and $r = n, g =2$ in Type $B$. Note that in Type $A$, to obtain \eqref{eq:typeaeul} one must multiply the formula in Theorem \ref{thm:coincidental} \eqref{7} by $t$.
\end{remark}

\begin{remark}\label{re:foulkes}
In \cite[Thm. 9]{Miller_Foulkes}, A. Miller shows that the change of basis matrix from the $z_{\ell}$-basis of $\E$ to the $\s^{(k)}$-basis is described by the transpose of the (reduced) Foulkes character matrix in Types $A$ and $B$ (as well as the complex reflection groups $G(m,1,n)$ for $m>2$; see Section \ref{sec:future}). Using Theorem \ref{thm:coincidental}, it is a straightforward to check that this surprising fact is true for $I_{2}(m)$ and $H_{3}$ as well. 
\end{remark}

\subsection{Connections to configuration spaces}\label{sec:config}
The final noteworthy feature of the Eulerian representations that we will discuss is their relationship to the cohomology of certain configuration spaces; recall that this was already known in Type $A$. 
\begin{definition}
The \emph{$n^{th}$ ordered configuration space} of a topological space $X$ is 
\[ \pconf_{n}(X): = \{ (x_{1}, \dots, x_{n}) \in X^{n}: x_{i} \neq x_{j} \textrm{ whenever } i \neq j \}. \]
\end{definition}
The case that $X = \R^{d}$ has been studied extensively. In \cite{arnold}, Arnol'd gave a presentation of $H^{*}\pconf_{n}(\R^{2})$; Fred Cohen extended this presentation to $H^{*}\pconf_{n}(\R^{d})$ for all $d\geq 2$ in \cite{cohen}. Importantly for our purposes,
when $d \geq 1$,
\[\pconf_{n}(\R^{d}) = \{ (x_{1}, \dots, x_{n}) \in \R^{dn}: x_{i} \neq x_{j} \text{ if } i \neq j \} = \oddcom_{\A(S_{n})}.  \]
Thus Theorem \ref{thm:coincidental} yields the following Corollary.

\begin{cor}\label{cor:typeaconfig}
For $0\leq k \leq n-1$ and odd $d \geq 3$,
\[ H^{k(d-1)}\pconf_{n}(\R^{d}) \cong_{S_{n}} \R S_{n} E_{n-1-k}\]
with $E_{n-1-k}$ as in Theorem \ref{thm:coincidental}.
\end{cor}
Since $H^{*}\pconf_{n}(\R^{d}) = H^{*}(\oddcom_{\A})$, the Eulerian representations give a complete description of every nonzero graded component of $H^{*}\pconf_{n}(\R^{d})$. Corollary \ref{cor:typeaconfig} was already known by comparing the computation of $H^{*}(\oddcom_{\A})$ in \cite[Thm. 4.4 (iii)]{sundaramwelker} with the character computations of the Type $A$ Eulerian idempotents by Hanlon in \cite{Hanlon}. In our framework, the proof of Corollary \ref{cor:typeaconfig} follows immediately from Theorem \ref{thm:coincidental}.

Theorem \ref{thm:coincidental} provides a similar description for the Type $B$ Eulerian idempotents. 
\begin{definition}\label{def:orbconfig}
For a group $G$ acting on a topological space $X$, the \emph{$n$-th orbit configuration space} is 
\[ \pconf_{n}^{G}(X) := \{(x_{1}, \dots, x_{n}) \in X^{n}: g\cdot x_{i} \cap g \cdot x_{j} = \emptyset \text{ for } i \neq j \textrm{ and any } g\in G \}. \]
\end{definition}

Take $G = \Z_{2}$ and $X = \R^{d}$ with the action by the generator of $\Z_{2}$ on $z \in \R^{d}$ mapping $z $ to $-z$.  In his thesis, Xicocencatl gives a description of $H^{*}\pconf_{n}^{\Z_{2}}(\R^{d})$ for $d$ of any parity \cite[Thm. 3.1.3]{xico}. Once again, for our purposes, it is enough to note that
\[ \pconf_{n}^{\Z_{2}}(\R^{d}) = \{ (x_{1}, \dots, x_{n}) \in \R^{dn} : x_{i} \neq \pm x_{j} \text{ for } i \neq j \text{ and } x_{i} \neq 0 \} = \oddcom_{\A(B_{n})} . \]
Thus Theorem \ref{thm:coincidental} immediately gives an equivariant description of each nonzero graded piece of $H^{*}\pconf_{n}^{\Z_{2}}(\R^{d})$. 

\begin{cor}\label{cor:typebconfig}
For $0 \leq k \leq n$ and odd $d \geq 3$,
\[H^{k(d-1)}\pconf^{\Z_{2}}_{n}(\R^{d}) \cong_{B_{n}} \R B_{n} E_{n-k} \]
with $E_{n-k}$ as in Theorem \ref{thm:coincidental}.
\end{cor}

For connections between $H^{*}\pconf_{n}(\R^{3})$ and $H^{*}\pconf^{\Z_{2}}_{n}(\R^{3})$ with polynomial factorizations over finite fields, see recent work by Hyde \cite{hyde} and Peterson-Tosteson \cite{petersen2020factorization}.

\section{Eulerian representations for finite Coxeter groups}\label{sec:allcox}
We now turn to the case that $W$ is an arbitrary finite Coxeter group with reflection arrangement $\A$. For our purposes, the key differences between coincidental and general reflection groups are that 
\begin{itemize}
    \item[(1)] in general $\A^{X}$ is not necessarily a reflection arrangement, and
    \item[(2)] if $\dim(X) = \dim(Y)$ for $X,Y \in \lat(\A)$, it is not necessarily true that $\sigma_{X}=\sigma_{Y}$.
\end{itemize}
Hence the eigenspaces of $\s$ will not effectively group together all flats of the same codimension as in the coincidental case. 

To combat this problem, we do two new things. First we introduce an element $\T \in \R W$ whose eigenspaces will be indexed by flat orbits $[X]$. Second, we will introduce a finer grading on $\VG(\A)$ by flat. 
\subsection{A Barr-like element}
\label{Barr-like-element-section}
We first define the element of $\R W$ whose eigenspaces will be indexed by flat orbits. 

\begin{definition}
Let 
\[ \T:= \sum_{T \subset S} \sum_{\substack{w \in W \\ \Des(w) \subset T }} c_{T} w= \sum_{T \subset S} c_{T} Y_{T} \]
where the collection of coefficients $\{ c_{T} \}_{T \subset S} \subset \R$ are positive and algebraically independent over $\Q$.
\end{definition}
Recalling that $\varphi: \R \F^{W} \to \dopp$ is the isomorphism in Theorem \ref{thm:bidigare}, applying $\varphi^{-1}$ to $\T$ gives:
\[ \tilde{\T} := \varphi^{-1}(\T)= \sum_{F \in \F} c_{[F]} F = \sum_{[F] \in \F^{W}} c_{[F]} \zeta_{[F]} \]
where $c_{[F]} = c_{T}$ for $F$ of type $T$.

\begin{prop}\label{prop:algindbarr}
The element $\T$ acts semisimply on $\R W$ with eigenvalues $\tau_{[X]}$ for $[X] \in \lat(\A)/W$, such that $\tau_{[X]} = \tau_{[Y]}$ if and only if $[X] = [Y]$. 
\end{prop}
\begin{proof}
Let $X, Y \in \lat$. 
It is clear that $\T$ acts semisimply on $\R W$ because $\tilde{\T}$ acts semisimply on $\R \ch$ as each $c_{[F]} > 0$. Suppose that each face of type $T\subset S$ lies in $X$ with multiplicity $m^{X}_{T}$ and in $Y$ with multiplicity $m_{T}^{Y}$. Hence 
\[\tau_{X} = \sum_{T \subset S} m_{T}^{X}c_{T}. \]
and similarly for $\tau_{Y}$. By construction $\tau_{X} = \tau_{Y}$ if $[X] = [Y]$. On the other hand, if $\tau_{X} = \tau_{Y},$ then because the coefficients $c_{T}$ are algebraically independent over $\Q$, it follows that $m^{X}_{T} = m^{Y}_{T}$ for every $T \subset S$. This forces $\dim(X) = \dim(Y)$. Moreover, for a maximal face $F$ in $X$, there must be a face $G \in Y$ such that $wF = G$ for some $w \in W$. Because $\dim(X) = \dim(Y)$ and $F$ is maximal, $G$ is maximal. Since the $W$-action commutes with $\mathfrak{s}$, 
\[ w X = \mathfrak{s}(wF) =\mathfrak{s}(G) = Y. \] 
Hence $[X] = [Y]$. 
\end{proof}
\subsection{A finer grading for $\mathcal{V}$}
We would like to give an interpretation of the decomposition of $\R W$ by $\T$ in terms of graded pieces of $\VG(\A)$. In order to do so, we must prove that $\VG(\A)$ admits a grading by flats. Our work will parallel the analogous result by Orlik-Terao \cite[Thm. 3.26, Cor 3.27]{orlikterao} that the Orlik-Solomon algebra has a flat decomposition.

Recall that $E = \R[e_{H_{i}}]_{H_{i} \in \A}$ and $\VG = E / \I$ (Definition \ref{def:vg}), where $\I$ is generated by 
\begin{enumerate}
    \item $e_{i}^{2}$ for every $H_{i} \in \A$ and,
    \item $\tpart(e_{C})$ for every circuit $C$ in $\A$.
\end{enumerate}

For $M = (H_{i_{1}}, \dots, H_{i_{k}})$, write $e_{M} = e_{i_{1}}\dots e_{i_{k}}$ and $\cap M = H_{i_{1}} \cap H_{i_{2}} \cap \dots \cap H_{i_{k}}$. Define for each $X \in \lat(\A)$ the $\R$-subspace $E_{X}:= \R\{ e_{M}: \cap M = X \}$. Note that $E$ has a vector space decomposition
\begin{equation}\label{eq:flatdecom} E = \bigoplus_{X \in \lat(\A)} E_{X}. \end{equation}
We will show that this decomposition holds for $\VG$ as well. 

\begin{theorem}\label{thm:vggrading}
For a central arrangement $\A$, there is a vector space decomposition
\[ \I = \bigoplus_{X \in \lat(\A)} \I \cap E_{X} \]
inducing the decomposition
\[ \VG(\A) = \bigoplus_{X \in \lat(\A)} \VG(\A)_{X},\]
where $\VG(\A)_{X}:= E_{X}/\I \cap E_{X}$.
\end{theorem}
\begin{proof}
It is clear that $\sum_{X \in \lat(\A)} \I \cap E_{X} \subset \I$. By \eqref{eq:flatdecom}, it is sufficient to show the other containment. Since $E$ is spanned by $e_{M}$, it will follow from the generating relations of $\I$ that $\I \subset \sum_{X \in \lat(\A)} \I \cap E_{X}$ if for any subset of hyperplanes $M$ in $\A$,
\begin{align*}
(1)& \hspace{1em} e_{i}^{2}e_{M} \in \sum_{X \in \lat(\A)} \I \cap E_{X} \textrm{ for every }H_{i} \in \A \textrm{ and,} \\
(2)& \hspace{1em} e_{M}\tpart(e_{C})\in  \sum_{X \in \lat(\A)} \I \cap E_{X} \textrm{ for every circuit C in }\A.
\end{align*}

The first condition is satisfied because $e_{i}^{2}e_{M} \in \I \cap E_{X}$ for $X = \cap (H_{i} \cup M)$. The second condition follows from the fact that if $C$ is a circuit, $\cap C = \cap C \backslash H_{i}$ for any $H_{i} \in C$. Hence $e_{M}\tpart(e_{C}) \in \I \cap E_{Y}$ for $Y = \cap (M \cup C)$.
\end{proof}
\begin{remark}
Since $\VG(\A)$ has an nbc-basis, Theorem \ref{thm:vggrading} implies that $\VG(\A)_{X}$ also has an nbc-basis of monomials $e_{M}$ indexed by nbc-sets $M$ with $\cap M = X$.
\end{remark}
Recall that $\A_{X}$ is the localized arrangement at $X$. We would like to relate $\VG(\A)_{X}$ to the top degree of $\VG(\A_{X})$.
\begin{prop}\label{prop:vgflat}
For any $X \in \lat(\A)$ with $\codim(X) =k$, there is an  $N_{X}$-equivariant isomorphism $\VG^{k}(\A_{X}) \cong_{N_{X}} \VG(\A)_{X}$.
\end{prop}
\begin{proof}
 Let $\mathfrak{i}$ be the natural inclusion $\mathfrak{i}: E(\A_{X}) \xhookrightarrow{} E(\A)$. 
The relations in $\I(\A_{X})$ hold in $\I(\A)$, and thus $\mathfrak{i}(\I(A_{X})) \subset \I(A)$, inducing a map 
 \[\mathfrak{i}: \VG(\A_{X}) \to \VG(\A). \]
Consider the image of $\mathfrak{i}(\VG^{k}(\A_{X}))$. Since both $\VG^{k}(\A_{X})$ and $\VG(\A)_{X}$ have a basis given by nbc-monomials $e_{M}$ with $\cap M = X$, it follows that
$\mathfrak{i}(\VG^{k}(\A_{X}))$ must map isomorphically onto $\VG(\A)_{X}$. The $N_{X}$ equivariance is clear from the definitions of $N_{X}$ and $\mathfrak{i}$.
\end{proof}

Recall that $N_{X}$ is the set-wise stabilizer of $X$. To index pieces of $\VG(\A)$ by flat orbits, define the $W$-module $\VG(\A)_{[X]} := \ind_{N_{X}}^{W} \VG(\A)_{X}$. The space $\VG(\A)_{[X]}$ will be used to describe the eigenspaces of $\T$ in the subsequent section.

\subsection{Results for arbitrary finite reflection groups}
Using the tools developed above, we may now give a description of the eigenspaces of $\T$ for any reflection group $W$. Recall that $(\R W)_{\tau_{[X]}}$ is the $\tau_{[X]}$-eigenspace of $\T$. 

Note that like $\tilde{\s}$, the element $\tilde{\T} = \varphi^{-1}(\T) \in \R \F^{W}$ acts semisimply and is separating, meaning that for $X > Y$ in $\lat(\A)$ one has $\tau_{[X]} < \tau_{[Y]}$. Thus $\tilde{\T}$ determines a unique family of flat-orbit idempotents coming from an invariant section\footnote{Recall that a family of flat orbit idempotents depends on the choice of invariant section. By contrast to the Eulerian idempotents considered in Section \ref{sec:coincidental}, the section in this context comes from $\T$, not the uniform section.} by \cite[Thm. 12.17]{aguiarmahajan}; write this family as 
$\{ \mathfrak{e}_{[X]} \}_{[X] \in \lat(\A)/W}$, where 
 \[ \tilde{\T} = \sum_{[X] \in \lat(\A)/W} \tau_{[X]}\mathfrak{e}_{[X]}. \]

\begin{theorem}\label{thm:allcox} For each $[X] \in \lat(\A)/W$, the following are isomorphic as $W$-representations:
\begin{enumerate}
    \item \label{a} $\VG(\A)_{[X]}$,
    \item\label{b} $\WH_{[X]}$,
    \item\label{c} $(\R W)_{\tau_{[X]}}$, and
    \item \label{d} $ \mathfrak{e}_{[X]}\R \ch$.
\end{enumerate}

\end{theorem}
\begin{proof}
\noindent
\begin{itemize}
    \item \eqref{a} if and only if \eqref{b}: 
Suppose $\codim(X) = k$. Since $N_{X}$ acts on $\A_{X}$, Theorem \ref{thm:moseley} implies that 
$H^{*}(\oddcom_{\A_{X}}) \cong_{N_{X}} \VG(\A_{X})$ when $d \geq 3$ and odd. Applying Theorem
\ref{thm:GM} (2) to the top cohomology group gives
\[ \WH_{X} \otimes \Det_{V/X} \cong_{N_{X}} \VG^{k}(\A_{X}). \]
Finally, apply Proposition \ref{prop:vgflat} and induce from $N_{X}$ to $W$.
    \item \eqref{b} if and only if \eqref{c}: follows immediately from Proposition \ref{prop:algindbarr} and Theorem \ref{thm:rsw}.
     \item \eqref{c} 
     if and only if \eqref{d}: follows from the fact that the $\frake_{[X]}$ are uniquely determined by $\tilde{\T}$ and $\varphi: \R \F^{W} \to \dopp$ is an isomorphism with $\varphi(\tilde{\T}) = \T$.
\end{itemize}
\end{proof}
\begin{remark}
By Theorems \ref{thm:coincidental} and \ref{thm:allcox} when $W$ is coincidental, as $W$-representations
\[ \bigoplus_{\substack{[X] \in \lat(\A)/W \\ \dim(X) = k}} (\R W)_{\tau_{[X]}}  \cong_{W} (\R W)_{\sigma_{k}} \]
despite the fact that 
\[\bigoplus_{\substack{[X] \in \lat(\A)/W \\ \dim(X) = k}} \varphi(\frake_{[X]}) \neq \varphi(\frake_{k}) \]
since the idempotents on the left-hand-side do not come from the uniform section, while the idempotents on the right-hand-side do. 
\end{remark}

\section{Future Directions} \label{sec:future}
\subsection{Description of the Eulerian representations}
In \cite{Hanlon}, Hanlon describes the characters of the Type $A$ Eulerian representations as a direct sum of induced representations of certain centralizers in $S_{n}$. Bergeron and Bergeron conjectured that there was a similar description in Type $B$ \cite[Rmk. 3.2]{Bergeron-Bergeron}, which N. Bergeron resolved in \cite{bergeron1995hyperoctahedral}. Our work raises more questions, potentially related to Conjecture 2.1 in Douglass-Pfeiffer-R{\"o}hrle \cite{douglass2014cohomology}.
\begin{question}
Is there a uniform description of either $ (\R W)_{\sigma_{k}}$ for $W$ coincidental or $(\R W)_{\tau_{[X]}}$ for any $W$ as (direct sums of) induced representations from centralizers in $W$? Is there a formula for their characters?
\end{question}
\subsection{Connections to Solomon's descent algebra}
In \cite{cellini1995general}, Cellini constructs commutative subalgebras\footnote{In fact, her subalgebras are in $\dopp$.} of $\Q W$ for every Weyl group $W$; in Type $C$, her subalgebra is semisimple and contains $\E(B_{n})$. A natural question is therefore:
\begin{question}\label{q:cellini}
Can the subalgebras defined by Cellini be described in terms of some family of flat-orbit idempotents $\{ \frake_{[X]}\}_{[X] \in\lat(\A)/W}$ and ``Barr-like'' element? 
\end{question}
Progress on Question \ref{q:cellini} could determine whether the subalgebras she introduces are semisimple for other types. Relatedly, $\R \T \subset \D$, the subalgebra generated by the ``Barr-like" element $\T$ from Section~\ref{Barr-like-element-section}, is also commutative.
\begin{question}
    Is $\R \T$ a maximal commutative subalgebra of $\D$?
\end{question}
\subsection{Extension to complex reflection groups}
Though we have only discussed real coincidental reflection groups, there is a notion of coincidental type for complex reflection groups as well. Nonreal reflection groups of coincidental type are exactly what have been called the (nonreal) {\it Shephard groups}, and share many properties with real coincidental reflection groups. See A. Miller \cite{Miller_Foulkes}, Reiner-Shepler-Sommers \cite{reiner2019invariant} for more detail. This motivates the following question:

\begin{question}
Can the Eulerian representations be generalized to all Shephard groups?
\end{question}
The complex reflection group $G(m,1,n) \cong \Z_{m} \wr S_{n}$ seems like a good place to start, due to promising work by A. Miller \cite{Miller_Foulkes} using results by Moynihan \cite{Moynihan} and Steingr\'{i}msson \cite{steingrimsson}.
\bibliographystyle{abbrv}
\bibliography{bibliography}
\end{document}